\documentclass{article}%
\usepackage{amssymb}
\usepackage{amsmath}
\usepackage{amsfonts}
\usepackage{graphicx}%
\newtheorem{theorem}{Theorem}[section]

\newtheorem{corollary}[theorem]{Corollary}

\newtheorem{definition}[theorem]{Definition}
\newtheorem{example}[theorem]{Example}

\newtheorem{lemma}[theorem]{Lemma}

\newtheorem{proposition}[theorem]{Proposition}

\newenvironment{proof}[1][Proof]{\noindent\textbf{#1.} }{\ \rule{0.5em}{0.5em}}
\begin{document}

\title{The Theory of Prime Ideals of Leavitt Path Algebras over Arbitrary Graphs}
\author{Kulumani M. Rangaswamy\\Department of Mathematics, University of Colorado\\Colorado Springs, Colorado 80918, USA\\E-mail: krangasw@uccs.edu}
\maketitle

\begin{abstract}
Given an arbitrary graph $E$ and a field $K$, the prime ideals as well as the
primitive ideals of the Leavitt path algebra $L_{K}(E)$ are completely
characterized in terms of their generators. \ The stratification of the prime
spectrum of $L_{K}(E)$ is indicated with information on its individual
stratum. Necessary and sufficient conditions are given on the graph $E$ under
which every prime ideal of the Leavitt path algebra $L_{K}(E)$ is primitive.
Leavitt path algebras with Krull dimension zero are characterized and those
with various prescribed Krull dimension are constructed. The minimal prime
ideals of $L_{K}(E)$ are described in terms of the graphical properties of $E$
and using this, complete descriptions of the height one as well as the
co-height one prime ideals of $L_{K}(E)$ are given.

\end{abstract}

\section{ Introduction}

The Leavitt path algebras were introduced in \cite{Abrams1} and \cite{Ara7} as
algebraic analogs of the $C^{\ast}$-algebras\ ([19]) and the study of their
algebraic structure has been the subject of a series of papers in recent years
(see, for e.g., \cite{Abrams1} - \cite{Aranda Pino 11}, \cite{Goodearl 13},
\cite{Tomforde 20}). In this paper we develop the theory of the prime ideals
of the Leavitt path algebras $L_{K}(E)$ for an arbitrary sized graph $E$. Here
the graph $E$ is arbitrary in the sense that no restriction is placed either
on the number of vertices in $E$ or on the number of edges emitted by a single
vertex in $E$ \ (such as row-finite or countable). We first give complete
characterizations,\ in terms of the generators, of the prime ideals as well as
the primitive ideals of the Leavitt path algebra $L_{K}(E)$. We also describe
the stratification of the prime spectrum $Spec(L_{K}(E))$ with information
about its individual stratum many of which are homeomorphic to
$Spec(K[x,x^{-1}])$. Our investigation shows that the non-graded prime ideals
of a Leavitt path algebra $L_{K}(E)$ are always primitive and, as also noted
in \cite{Aranda Pino 9}, there are graded prime ideals of $L_{K}(E)$ which are
not primitive. This is in contrast to the situation for the graph C*-algebras
in which there is no distinction between the prime and the primitive ideals
(see \cite{Dixmier 12},\cite{Hong 14}, \cite{Raeburn 19}). Using our
characterization of right (= left) primitive ideals of a Leavitt path algebra,
we give necessary and sufficient conditions on the graph $E$ under which every
prime ideal of $L_{K}(E)$ is primitive. Examples are constructed to illustrate
the different possibilities for the prime and the primitive spectrum of
$L_{K}(E)$. In the case when $E$ is a row-finite graph, it was shown in
\cite{Aranda Pino 9} \ that a bijection exists between the set of prime ideals
of the Leavitt path algebra $L_{K}(E)$ \ and a certain set which involves
subsets of vertices (called maximal tails ) and the prime spectrum of
$K[x,x^{-1}]$. This was analogous to the work done in \cite{Hong 14} for graph
$C^{\ast}$- algebras. In this paper, we extend and sharpen the bijective
correspondence of \cite{Aranda Pino 9} to the case when $E$ is an
arbitrary-sized graph. There are a number of applications of our main result
(Theorem \ref{Main Theorem}). First we describe those graphs $E$ for which
every non-zero prime ideal of $L_{K}(E)$ is maximal and this leads to the
characterization the Leavitt path algebras whose Krull dimension is zero.
Examples are constructed of Leavitt path algebras $L_{K}(E)$ with various
prescribed finite or infinite Krull dimension. As another application, we show
that a graph $E$ satisfies the Condition (K) if and only if every prime ideal
of $L_{K}(E)$ is graded. Next we describe the minimal prime ideals of a
Leavitt path algebra. The height $1$ prime ideals play an impotant part in the
study of commutative rings and algebraic geometry. As a first step to examine
their role in the study of Leavitt path algebras, we characterize not only the
height $1$ prime ideals but also the co-height $1$ prime ideals of $L_{K}(E)$
in terms of the graphical properties of $E$. Finally, we also consider when a
prime homomorphic image of a Leavitt path algebra is again a Leavitt path algebra.

\section{Preliminaries}

For the general notation and terminology used in this paper, we refer to
\cite{Abrams1},\cite{Aranda Pino 9} and \cite{Raeburn 19}. \ We give a short
description of some of the concepts that we will be using. A directed graph
$E=(E^{0},E^{1},r,s)$ consists of a set $E^{0}$ of vertices, a set $E^{1}$ of
edges and maps $r,s$ from $E^{1}$ to $E^{0}$. For each $e\in E^{1}$, $s(e)=u$
is called the \textit{source} of $e$ and $r(e)=v$ the \textit{range} of $e$
and $e$ is called an edge from $u$ to $v$. \ All the graphs \ that we consider
in this note are arbitrary. \ A vertex $v$ is called a \textit{sink} if it
emits no edges. It is called a \textit{regular vertex} if $0<|s^{-1}%
(v)|<\infty$. If $|s^{-1}(v)|=\infty$, we say $v$ is an \textit{infinite
emitter}.

Given a graph $E$, ($E^{1})^{\ast}$ denotes the set of symbols $e^{\ast}$ one
for each $e\in E^{1}$ called the ghost edges.

DEFINITION: Let $E$ be a directed graph and $K$ be any field. The
\textit{Leavitt path algebra} $L_{K}(E)$ of the graph $E$ with coefficients in
$K$ is the $K$-algerbra generated by a set $\{v:v\in E^{0}\}$ of pairwise
orthogonal idempotents together with a set of variable $\{e,e^{\ast}:e\in
E^{1}\}$ which satisfy the following conditions:

(1) \ $s(e)e=e=er(e)$ for all $e\in E^{1}$.

(2) $r(e)e^{\ast}=e^{\ast}=e^{\ast}s(e)\infty$\ for all $e\in E^{1}$.

(3) (The "CK-1 relations") For all $e,f\in E^{1}$, $e^{\ast}e=r(e)$ and
$e^{\ast}f=0$ if $e\neq f$.

(4) (The "CK-2 relations") For every regular vertex $v\in E^{0}$,
\[
v=\sum_{e\in E^{1},s(e)=v}ee^{\ast}.
\]

A \ \textit{path} $\mu$ is a finite sequence of edges $e_{1}...e_{n}$ with
$r(e_{i})=s(e_{i+1})$ for all $i=1,...n-1$. $\mu^{\ast}$ denotes the
corresponding ghost path $e_{n}^{\ast}...e_{1}^{\ast}$. The path $\mu$ is said
to be a \textit{closed path} if $r(e_{n})=s(e_{1})$. A closed path
$e_{1}...e_{n}$ is said to be a \textit{cycle based at }$v$ if $s(e_{1})=v$
and $s(e_{i})\neq s(e_{j})$ for $i\neq j$.

Every element in a Leavitt path algebra $L_{K}(E)$ can be written in the form
$%
{\textstyle\sum\limits_{i=1}^{n}}
k_{i}\alpha_{i}\beta_{i}^{\ast}$ where $k_{i}\in K$, $n$ a positive integer
and $\alpha,\beta$ are paths in $E$ (see [\cite{Abrams1}])

An edge $f$ is called an \textit{exit} to a path $e_{1}...e_{n}$ if there is
an $i$ such that $s(f)=s(e_{i})$ and $f\neq e_{i}$. A graph $E$ satisfies
\textit{Condition (L)} if every cycle in $E$ has an exit. The graph $E$ is
said to satisfy the \textit{Condition (K)} if no vertex\textit{\ }$v$ in $E$
is the base of precisely one cycle, that is, either no cycle is based at $v$
or at least two are based at $v$.

The following well-known result turns out to be very useful in our
investigation: If there exists a cycle $c$ without exits and is based at a
vertex $v$ in a graph $E$, then the subring $vL_{K}(E)v\cong K[x,x^{-1}]$
under an isomorphism that maps $v$ to $1$, $c$ to $x$ and $c^{\ast}$ to
$x^{-1}$. \ Indeed this isomorphism is evident if one notices that a typical
element of $vL_{K}(E)v$ \ is a $K$-linear combination of elements of the form
$v\alpha\beta^{\ast}v$ which simplifies, since $c$ has no exits, to $v$%
,$c^{n}$ or $(c^{\ast})^{n}$ for some integer $n$.

If there is a path from a vertex $u$ to a vertex $v$, we write $u\geq v$. \ A
subset $H$ of vertices is called a \textit{hereditary} set if whenever $u\in
H$ and $u\geq v$ for some vertex $v$, then $v\in H$. A set of vertices $H$ is
said to be \textit{saturated} if, for any regular vertex $v$, $r(s^{-1}%
(v))\subset H$ implies $v\in H$. If $I$ is a two-sided ideal of $L_{K}(E)$, it
is easy to see that $I\cap E^{0}$ is a hereditary saturated subset of $E^{0}$.

For every non-empty subset $X$ of vertices in a graph $E$, we can define
\textit{the restricted subgraph }$E_{X}$ where $(E_{X})^{0}=X$ and
$(E_{X})^{1}=\{e\in E^{1}:s(e),r(e)\in X\}$.

We shall be making extensive use of the following concepts and results from
\cite{Tomforde 20}. A vertex $w$ is called a \textit{breaking vertex }of a
hereditary saturated subset $H$ if $w\in E^{0}\backslash H$ is an infinite
emitter with the property that $1\leq|r(s^{-1}(v))\cap(E^{0}\backslash
H)|<\infty$. The set of all breaking vertices of $H$ is denoted by $B_{H}$.
For any $v\in B_{H}$, $v^{H}$ denotes the element $v-\sum_{s(e)=v,r(e)\notin
H}ee^{\ast}$. Given a hereditary saturated subset $H$ and a subset $S\subset
B_{H}$, $(H,S)$ is called an admissible pair and $I_{(H,S)}$ denotes the ideal
generated by $H\cup\{v^{H}:v\in S\}$. It was shown in \cite{Tomforde 20} that
the graded ideals of $L_{K}(E)$ are precisely the ideals of the form
$I_{(H,S)}$ for some admissible pair $(H,S)$. \ Moreover, it was shown that
$I_{(H,S)}\cap E^{0}=H$ and $I_{(H,S)}\cap B_{H}=S$.

Given an admissible pair $(H,S)$, the corresponding \textit{Quotient Graph}
$E\backslash(H,S)$ is defined as follows:%
\begin{align*}
(E\backslash(H,S))^{0}  &  =(E^{0}\backslash H)\cup\{v^{\prime}:v\in
B_{H}\backslash S\};\\
(E\backslash(H,S))^{1}  &  =\{e\in E^{1}:r(e)\notin H\}\cup\{e^{\prime}:e\in
E^{1},r(e)\in B_{H}\backslash S\}
\end{align*}
Further $r$ and $s$ are extended to $(E\backslash(H,S))^{0}$ by setting
$s(e^{\prime})=s(e)$ and $r(e^{\prime})=r(e)^{\prime}$.

Notice that the elements $v^{\prime}$, where $v\in B_{H}\backslash S$, are all
sinks in the graph $E\backslash(H,S)$. \ Theorem 5.7 of \cite{Tomforde 20}
states that there is an epimorphism $\phi:L_{K}(E)\rightarrow L_{K}%
(E\backslash(H,S))$ with $\ker\phi=$ $I_{(H,S)}$ and that $\phi(v^{H}%
)=v^{\prime}$ for $v\in B_{H}\backslash S$. Thus $L_{K}(E)/I_{(H,S)}\cong
L_{K}(E\backslash(H,S))$. (This theorem has been established in \cite{Tomforde
20} under the hypothesis that $E$ is a graph with at most countably many
vertices and edges; however, an examination of the proof reveals that the
countability condition on $E$ is not utilized. So the Theorem 5.7 of
\cite{Tomforde 20} holds for arbitrary graphs $E$)

Let $H$ be a non-empty hereditary saturated set of vertices in a graph $E$. A
subset $M$ of $H$ is said to be a \textit{maximal tail }in\textit{\ }$H$
\textit{(see [9], [14]) }if it satisfies the following conditions:

(MT-1) If $v\in M$ and $u\in H$ with $u\geq v$, then $u\in M;$

(MT-2) If $v\in M$ is a regular vertex, then there is an $e\in E^{1}$ with
$s(e)=v$ and $r(e)\in M$;

(MT-3) For any two $u,v\in M$ there exists $w\in M$ such that $u\geq w$ and
$v\geq w$.

It is easy to see that $M\subset E^{0}$ satisfies both the (MT-1) and the
(MT-2) conditions if and only if $E^{0}\backslash M$ is a hereditary saturated
subset of $E^{0}$.

Given a vertex $v$ in a graph $E$, we attach two special sets of vertices.

\textbf{Define} $T(v)=\{w\in E^{0}:v\geq w\}$ and $M(v)=\{w\in E^{0}:w\geq
v\}$.

The set $T(v)$ (called the \textit{tree of }$v$ in the literature) is the
smallest hereditary set containing $v$. The set $M(v)$ is maximal tail
containing $v$ whenever $v$ is a sink, an infinite emitter or a regular vertex
lying on a cycle in $M(v)$. But when $v$ is a regular vertex not lying on a
cycle, $M(v)$ satisfies the MT-1 and MT-3 conditions , but may not satisfy the
MT-2 condition.

In this paper "ideal" means "two-sided ideal". An ideal $P$ of a ring $R$ is
called a left (right) \textit{primitive ideal} if it is the left (right)
annihilator of a simple left (right) $R$-module. A ring $R$ is called a left
(right) \textit{\ primitive ring} if $\{0\}$ is a right (left) primitive
ideal. \ 

Recall that an ideal $P$ of a ring $R$ is called a \textit{prime ideal} if,
given any ideals $A,B$ of $R$, $AB\subset P$ implies that either $A\subset P$
or $B\subset P$. \ If $R$ is a graded ring, graded by a group (such as the
Leavitt path algebra $L_{K}(E)$), then it was shown in (\cite{NV 18},
Proposition II.1.4) that a graded ideal $P$ will be a prime ideal, if $P$
satisfies the above property only for graded ideals $A$ and $B$. \ This
observation will be used in the sequel.

\section{Prime ideals of $L_{K}(E)$}

In this section, we give a complete characterization of the prime ideals of a
Leavitt path algebra $L_{K}(E)$ of an arbitrary graph $E$. In \cite{Aranda
Pino 9} graded prime ideals of $L_{K}(E)$ were described for row-finite graphs
$E$. We extend this result to the case when $E$ is an arbitrary graph.
Moreover, we also characterize the non-graded prime ideals of $L_{K}(E)$ for
arbitrary sized graphs $E$ by means of their generating sets. As will be clear
from subsequent sections, this helps in our deeper study of the prime ideals
of Leavitt path algebras. The bijective correspondence established in
\cite{Aranda Pino 9} for prime ideals of $L_{K}(E)$ for a row-finite graph $E$
is also extended to the case when $E$ is an arbitrary graph.

We shall be using the following description of a prime Leavitt path algebra
$L_{K}(E)$ that was established \ for row-finite graphs $E$ in \cite{Aranda
Pino 9} and for arbitrary graphs $E$ in \cite{abrams3}. I underestand that M.
Siles Molina also has independently obtained the following charaterization.

\begin{theorem}
\label{primeLPA}(\cite{abrams3}, \cite{Aranda Pino 9}) Let $E$ be an arbitrary
graph and $K$ be any\ field. Then the Leavitt path algebra $L_{K}(E)$ is a
prime ring if and only if $E^{0}$ satisfies the MT-3 condition.
\end{theorem}

The following theorem from \cite{abrams4} is used in the sequel.

\begin{theorem}
\label{IdealThm}(\cite{abrams4}) Let $E$ be an arbitrary graph and $K$ be any
field. Then any non-zero ideal of the $L_{K}(E)$ is generated by elements of
the form
\[
(u+%
{\textstyle\sum\limits_{i=1}^{k}}
k_{i}g^{r_{i}})(u-%
{\textstyle\sum\limits_{e\in X}}
ee^{\ast})
\]
where $u\in E^{0}$, $k_{i}\in K$, $r_{i}$ are\ positive integers, $X$ is a
finite (possibly empty) proper subset of $s^{-1}(u)$ and, whenever $k_{i}%
\neq0$ for some $i$, then $g$ is a unique cycle based at $u$.
\end{theorem}

\bigskip Before proving our main theorem, we shall establish a series of
useful Lemmas. The next Lemma specializes Theorem \ref{IdealThm} to the case
of ideals containing no vertices.

\begin{lemma}
\label{Lermma3.3}Suppose $E$ is an arbitrary graph and $K$ is any field. If
$N$ is a non-zero ideal of $L_{K}(E)$ which does not contain any vertices of
$E$, then $N$ is generated by elements of the form $y=(u+%
{\textstyle\sum\limits_{i=1}^{n}}
k_{i}g^{r_{i}})$ where $g$ is a cycle without exits based at a vertex $u$.
\end{lemma}

\begin{proof}
From Theorem \ref{IdealThm}, we know that $N$ is generated by elements of the
form $y=(u+%
{\textstyle\sum\limits_{i=1}^{n}}
k_{i}g^{r_{i}})(u-%
{\textstyle\sum_{e\in X}}
ee^{\ast})\neq0$ where $g$ is a cycle in $E$ based at a vertex $u$ and where
$X$ is a finite proper subset of $s^{-1}(u)$.

Our first step is to show that the cycle $g$ has no exits and that $X$ must be
an empty set. Suppose $f\in s^{-1}(u)\backslash X$ with $r(f)=w$. If $f$ is
not the initial edge of $g$, then $f^{\ast}g=0$ and $f^{\ast}yf=f^{\ast}(u+%
{\textstyle\sum\limits_{i=1}^{n}}
k_{i}g^{r_{i}})f=f^{\ast}uf=r(f)=w\in N$, contradicting the fact that $N$
contains no vertices. If $f$ is the initial edge of $g$ say $g=f\alpha$, then
$f^{\ast}yf=w+%
{\textstyle\sum\limits_{i=1}^{n}}
k_{i}h^{r_{i}}\in N$ where $w=r(f)$ and $h$ is the cycle $\alpha f$. Then
$\alpha^{\ast}(w+%
{\textstyle\sum\limits_{i=1}^{n}}
k_{i}h^{r_{i}})\alpha=u+%
{\textstyle\sum\limits_{i=1}^{n}}
k_{i}g^{r_{i}}\in N$. If there is an exit $e$ at a vertex $u^{\prime}$ on $g$
and if $\beta$ is the part of $g$ connecting $u$ to $u^{\prime}$ (where we
take $\beta=u$ if $u^{\prime}=u$) and $\gamma$ is the part of $g$ from
$u^{\prime}$ to $u$ ( so that $g=\beta\gamma$), then, denoting $\gamma\beta$
by $d$, we get $e^{\ast}\beta^{\ast}(u+%
{\textstyle\sum\limits_{i=1}^{n}}
k_{i}g^{r_{i}})\beta e=e^{\ast}(u^{\prime}+%
{\textstyle\sum\limits_{i=1}^{n}}
k_{i}d^{r_{i}})e=r(e)\in N$, a contradiction. Thus the cycle $g$ has no exits.
In particular, $|s^{-1}(u)|=1$ and this implies, as $y\neq0$, that $X$ is an
empty set.Thus the generators of $N$ are of the form $y=(u+%
{\textstyle\sum\limits_{i=1}^{n}}
k_{i}g^{r_{i}})$.
\end{proof}

\begin{lemma}
\label{Lemma 3.4}Let $E$ be an arbitrary graph and $K$ be any field. Suppose
$E^{0}$satisfies the MT-3 condition. If $N$ is a non-zero ideal of $L_{K}(E)$
which does not contain any vertices of $E$, then there is a unique cycle $c$
without exits in $E$ and $N$ is a principal ideal generated by $p(c)$, where
$p(x)$ is a polynomial belonging to $K[x]$.
\end{lemma}

\begin{proof}
By Lemma \ref{Lermma3.3}, $N$ is generated by elements of the form $y=(u+%
{\textstyle\sum\limits_{i=1}^{n}}
k_{i}g^{r_{i}})$ where $g$ is a cycle without exits based at $u$. Since the
ideal $N$ does not contain any vertices, it is well known (see, for eg.,
Proposition 18 of \cite{abrams4} or Proposition 2.8 (ii) of \cite{Aranda Pino
9}) that $E$ does not satisfy the Condition (L). So there is a cycle $c$
without exits and based at a vertex $v$ in $E$. Now the MT-3 condition on
$E^{0}$ implies that $c$ is the only cycle without exits in $E$ (except
possibly a permutation of its vertices). This means that the cycle $g$ based
at $u$ is the same as the cycle $c$ based at $v$, obtained possibly by a
rotation of the vertices on $c$. If $\alpha$ is the part of $c$ from $v$ to
$u$ and $\beta$ is the part of $c$ from $u$ to $v$ (so that $c=\alpha\beta$
and $g=\beta\alpha$) , then $\beta^{\ast}(u+%
{\textstyle\sum\limits_{i=1}^{n}}
k_{i}g^{r_{i}})\beta=v+%
{\textstyle\sum\limits_{i=1}^{n}}
k_{i}c^{r_{i}}\in N$ and $\alpha^{\ast}(v+%
{\textstyle\sum\limits_{i=1}^{n}}
k_{i}c^{r_{i}})\alpha=$ $u+%
{\textstyle\sum\limits_{i=1}^{n}}
k_{i}g^{r_{i}}=y\in N$. From this it is clear that we can select \ a
generating set for the ideal $N$ consisting of elements of the form \ $(v+%
{\textstyle\sum\limits_{i=1}^{n}}
k_{i}c^{r_{i}})$ with this fixed cycle $c$ based at $v$, where $k_{i}\neq0$
for at least one $i$. We shall denote the element $(v+%
{\textstyle\sum\limits_{i=1}^{n}}
k_{i}c^{r_{i}})$ by $f(c)$, where $f(x)=1+%
{\textstyle\sum\limits_{i=1}^{n}}
k_{i}x^{r_{i}}\in K[x]$ is a polynomial of positive degree (and where we use
the convention that $c^{0}=v$). Let $p(x)$ be a polynomial of the
smallest\ positive degree in $K[x]$ such that $p(c)\in N$. By the division
algorithm in $K[x]$, every generator $f(c)$ of $N$ is easily shown to be a
multiple of $p(c)$. This proves that $N$ is the principal ideal generated by
$p(c)$.
\end{proof}

\begin{lemma}
\label{Lemma 3.5}Let $E$ be an arbitrary graph and $K$ be any field. Let $P$
be an ideal of $L_{K}(E)$ with $H=P\cap E^{0}$. Let $S=\{v\in B_{H}:v^{H}\in
P\}$. Then\ the graded ideal $I_{(H,S)}\subset P$ contains every other graded
ideal of $L_{K}(E)$ inside $P$.
\end{lemma}

\begin{proof}
Suppose $A=I_{(H_{1},S_{1})}\subset P$. We claim that $A\subset I_{(H,S)}$.
Clearly $H_{1}\subset E^{0}\cap P=H$. Note that if $H_{1}$ is empty,
then\ clearly $A=I_{(H_{1},S_{1})}=\{0\}\subset I_{(H,S)}$. So we assume that
$H_{1}$ is non-empty. We need to show that if $v\in S_{1}$ then $v^{H_{1}}\in
I_{(H,S)}$. Thus $v$ is a breaking vertex for $H_{1}$ and let $e_{1},..,e_{n}$
be the finitely many edges satisfying $s(e_{i})=v$ and $r(e_{i})\notin H_{1}$.
Suppose $\ v$ is a breaking vertex for $H$. By re-indexing, we may then assume
that for some $m\leq n$, $r(e_{i})\notin H$ for $i=1,...m_{\text{ },}$and that
$r(e_{j})\in H$ for $j=m+1,...,n$. Since $e_{j}\in I_{(H,S)}$ for
$j=m+1,...,n$, we get $v^{H_{1}}=v^{H}-%
{\textstyle\sum\limits_{i=m+1}^{n}}
e_{i}e_{i}^{\ast}\in I_{(H,S)}$, as desired. Suppose $v$ is not a breaking
vertex for $H$. Since $v$ is a breaking vertex for $H_{1}$ but not for $H$, we
have $r(s^{-1}(v))\subset H$. So $e_{i}=e_{i}r(e_{i})\in I_{(H,S)}$ for
$i=1,...,n$. As $%
{\textstyle\sum\limits_{i=1}^{n}}
e_{i}e_{i}^{\ast}\in I_{(H,S)}\subset P$ and $v-%
{\textstyle\sum\limits_{i=1}^{n}}
e_{i}e_{i}\in P$, we then conclude that $v\in P\cap E^{0}=H$. Clearly then
$v^{H_{1}}\in I_{(H,S)}$ and so $A\subset I_{(H,S)}$.
\end{proof}

\begin{lemma}
\label{Lemma 3.6}Let $P$ be a prime ideal of $L_{K}(E)$ with $H=P\cap E^{0}$
and let $S=\{v\in B_{H}:v^{H}\in P\}$. Then the ideal $I_{(H,S)}$ is also a
prime ideal of $L_{K}(E)$.
\end{lemma}

\begin{proof}
Suppose $A=I_{(H_{1},S_{1})}$ and $B=I_{(H_{2},S_{2}})$ are two graded ideals
of $L_{K}(E)$ with $AB\subset I_{(H,S)}$. Since $P$ is prime, one of them, say
$A$, is contained in $P$. \ By Lemma \ref{Lemma 3.5}, $A\subset I_{(H,S)}$.
Thus the ideal $I_{(H,S)}$ is a graded prime. Since $L_{K}(E)$ is graded by
the group $%
\mathbb{Z}
$, as noted in the Preliminaries section, we can appeal to (\cite{NV 18},
Proposition II.1.4) to conclude that $I_{(H,S)}$ is actually a prime ideal of
$L_{K}(E)$.
\end{proof}

The following consequence of Lemma \ref{Lemma 3.6} may be of some interest.

\begin{corollary}
\label{Cor 3.7}Let $E$ be an arbitrary graph and $K$ be any field. Then the
Leavitt path algebra \ $L_{K}(E)$ is a prime ring if and only if there is a
prime ideal of $L_{K}(E)$ which does not contain any vertices.
\end{corollary}

\begin{proof}
Suppose $P$ is prime ideal of $L_{K}(E)$ which does not contain any vertices.
If $P$ is a graded ideal, then by \cite{Tomforde 20}, $P=I_{(H,S)}$ where
$H=P\cap E^{0}$ and $S\subset B_{H}$. Since $P$ contains no vertices, $H$ and
$S$ are both empty sets and so the prime ideal $P=\{0\}$, proving that
$L_{K}(E)$ is a prime ring. Suppose the prime ideal $P$ is not graded and
$H=P\cap E^{0}$. By Lemma 3.6, the ideal $I=I_{(H,B_{H})}$ is then a (graded)
prime ideal. Since $I\subset P$ and $I\cap E^{0}=H$ by \cite{Tomforde 20} and
since $P$ contains no vertices, $H$ and hence $B_{H}$ must then be empty sets.
So $\{0\}=I$ is a prime ideal, showing that $L_{K}(E)$ is a prime ring. The
converse is obvious.
\end{proof}

The next Lemma is useful in our investigation and is, perhaps, known. The
proof is immediate if one uses the fact (see for eg., Proposition 10.2 of
\cite{Lam 16}) that an ideal $P$ of a not necessarily unital ring $R$ is a
prime ideal if $P\neq R$ and whenever $aRb\subset P$ for some elements $a,b\in
R$, then either $a\in P$ or $b\in P$.

\begin{lemma}
\label{Lemma 3.8} Let $R$ be a not necessarily unital ring and let $P$ be a
prime ideal of $R$. Then for any idempotent $v\in R$, $vPv$ is a prime ideal
of $vRv$.
\end{lemma}

For the sake of convenience, we introduce the following definition.

\begin{definition}
\label{Cycle without K} A cycle $c$ in a graph $E$ is called a \textbf{cycle
without K}\textit{, }if no vertex on $c$ is the base of another distinct cycle
in $E$ (where distinct cycles possess different sets of edges and different
sets of vertices). The set of all cycles without K in the graph $E$ is denoted
by $C(E)_{\kappa}$.
\end{definition}

We are now ready to prove our main result.

In the following $<X,Y,a>$ denotes the ideal generated by $X\cup Y\cup\{a\}$.
Also recall (from Preliminaries) that for any vertex $v$ we define
$M(v)=\{w\in E^{0}:w\geq v\}$.

\begin{theorem}
\label{Main Theorem} Let $E$ be an arbitrary graph and $K$ be any field. Let
$P$ be an ideal of $L_{K}(E)$ with $P\cap E^{0}=H$. Then $P$ is a prime ideal
of $L_{K}(E)$ if and only if $P$ satisfies one of the following conditions:

(i) \ $P=<H,\{v^{H}:v\in B_{H}\}>$ (where $B_{H}$ may be empty) and
$E^{0}\backslash H$ satisfies the MT-3 condition;

(ii) $P=<H,\{v^{H}:v\in B_{H}\backslash\{u\}\}>$ for some $u\in B_{H}$ (hence
$B_{H}$ is non-empty) and $E^{0}\backslash H=M(u)$;

(iii) $P=<H,\{v^{H}:v\in B_{H}\},f(c)>$ where $c$ is a cycle without K in $E$
based at a vertex $v$, $E^{0}\backslash H=M(v)$ and $f(x)$ is an irreducible
polynomial in $K[x,x^{-1}]$.
\end{theorem}

\begin{proof}
Now $H=P\cap E^{0}$. Let $S=\{w\in B_{H}:w^{H}\in P\}$.

Case 1: Suppose $P$ is a graded ideal. Then by (Theorem 5.7, \cite{Tomforde
20}), $P=I_{(H,S)}=<H,\{v^{H}:v\in S\}>$ and that $L_{K}(E)/P\overset{\phi
}{\cong}L_{K}(E\backslash(H,S))$. Thus $P$ is a prime ideal of $L_{K}(E)$ if
and only if $L_{K}(E\backslash(H,S))$ is a prime ring. Theorem \ref{primeLPA}
shows that this is equivalent to $E\backslash(H,S)^{0}$ satisfying the MT-3
condition. Now for each $u\in B_{H}\backslash S$, the corresponding vertex
$u^{\prime}$ is a sink in the graph $E\backslash(H,S)$. In view of the MT-3
condition, there can be at most one sink in $E\backslash(H,S)$. So
$B_{H}\backslash S$ is either empty or a singleton $\{u\}$. Hence $P$ is a
prime ideal $L_{K}(E)$ if and only if either $B_{H}=$ $S$ in which case
$E\backslash(H,S)^{0}=E^{0}\backslash H$ satisfies the MT-3 condition or
$B_{H}\backslash\{u\}=S$ in which case, $(E\backslash(H,S))^{0}=(E^{0}%
\backslash H)\cup\{u^{\prime}\}$ and $w\geq u^{\prime}$ for all $w\in
E\backslash(H,S)^{0}$, equivalently, that for all $w\in E^{0}\backslash H$,
$w\geq u$. Thus the primeness of the graded ideal $P$ is equivalent to
condition (i) or (ii).

Case 2: Let $P$ be a non-graded ideal. Suppose $P$ is prime. By Lemma
\ref{Lemma 3.6}, the graded ideal $I_{(H,S)}$ is also a prime ideal of
$L_{K}(E)$ contained in $P$. So, as proved in Case 1, either (i) $\ B_{H}=$
$S$ and $E^{0}\backslash H$ satisfies the MT-3 condition, or (ii)
$B_{H}\backslash\{u\}=S$ and, for all $w\in E^{0}\backslash H$, $w\geq u$. As
noted earlier, $L_{K}(E)/I_{(H,S)}\overset{\phi}{\cong}L_{K}(E\backslash
(H,S))$. Let $\phi(P/I_{(H,S)})=$ \ $N$. Note that, under condition (i),
$(E\backslash(H,S))^{0}=(E^{0}\backslash H)$ and, under condition (ii),
$(E\backslash(H,S))^{0}=(E^{0}\backslash H)\cup\{u^{\prime}\}$ and $w\geq
u^{\prime}$ for all $w\in(E\backslash(H,S))^{0}$, where $u^{\prime}=\phi
(u^{H}+I_{(H,S)})$. Note that $u^{H}\notin P$, as $u\notin S$ and so
$u^{\prime}\notin N$. Thus, in either case, it is clear that the non-zero
ideal $N$ of $L_{K}(E\backslash(H,S))$ does not contain any vertices. Since
$(E\backslash(H,S))^{0}$ satisfies the MT-3 condition, Lemma \ref{Lemma 3.4}
implies that there is a cycle $c$ based at a vertex $v$ and without exits in
the graph $E\backslash(H,S)$ such that $N$ is the ideal generated by $f(c)$
for some polynomial $f(x)\in K[x]$. \ We now claim that $\ B_{H}=S$ (that is,
condition (ii) \ of the graded prime ideal is not possible). Because if
$B_{H}\backslash\{u\}=S$, then $w\geq u^{\prime}$ for all $w$ in
$(E\backslash(H,S))^{0}$ and since $c$ has no exits, $u^{\prime}$ must lie on
the cycle $c$. But this is impossible, since $u^{\prime}$ is a sink in
$E\backslash(H,S)$. \ So $B_{H}=S$ must hold. Putting all these facts
together, we conclude that $P$ is the ideal generated by $H\cup\{u^{H}:u\in
B_{H}\}\cup\{f(c)\}$. Now $(E\backslash(H,S))^{0}=(E\backslash(H,B_{H}%
))^{0}=(E^{0}\backslash H)$ satisfies the MT-3 condition and contains $c^{0}$.
Clearly then $(E\backslash(H,S))^{0}=E^{0}\backslash H$ $=\{w\in E^{0}:w\geq
v\}=M(v)$, where $v$ is the base of the cycle $c$. It is also clear that $c$
is a cycle without K in $E$. To complete the proof, we need only to show that
$f(x)$ is irreducible. Now $N$ is a prime ideal of $L_{K}(E\backslash(H,S))$
and hence, by Lemma \ref{Lemma 3.8}, $vNv$ is a (non-zero) prime ideal of
$vL_{K}(E\backslash(H,S))v\overset{\theta}{\cong}K[x,x^{-1}]$ generated by
$vf(c)v=f(c)$. Here the isomorphism $\theta$ maps $f(c)$ to $f(x)$ as it maps
$v$ to $1$, $c$ to $x$ and $c^{\ast}$ to $x^{-1}$(as noted in the
Preliminaries). Since $f(x)$ generates the non-zero prime (hence maximal)
ideal $\phi(vNv)$ in the Euclidean domain $K[x,x^{-1}]$, $f(x)$ must\ then be
an irreducible polynomial in $K[x,x^{-1}]$.

Conversely, suppose (a) $E$ contains a cycle $c$ without K and based at a
vertex $v$, (b) $E^{0}\backslash H=M(v)$ and (c) there exists an irreducible
polynomial $f(x)\in K[x,x^{-1}]$ such that $P$ is the ideal generated by
$\{f(c)\}\cup I_{(H,B_{H})}$. Now hypothesis (b) implies $(E\backslash
(H,B_{H}))^{0}=E^{0}\backslash H=$ $M(v)$. So $E\backslash(H,B_{H})$ satisfies
the MT-3 condition and contains the cycle $c$. As $c$ is a cycle without K,
the MT-3 condition implies that $c$ has no exits in the graph $E\backslash
(H,B_{H})$. Now, by (Theorem 5.7, \cite{Tomforde 20}), $L_{K}(E)/I_{(H,B_{H}%
)}\overset{\phi}{\cong}L_{K}(E\backslash(H,B_{H}))$. If $N=\phi(P/I_{(H,B_{H}%
)})$, then, by hypothesis (c), the ideal $N$ is generated by $f(c)$. As $f(x)$
is an irreducible polynomial in $K[x,x^{-1}]\overset{\theta^{-1}}{\cong}%
vL_{K}(E\backslash(H,B_{H}))v$, the ideal $vNv$, being generated by
$vf(c)v=f(c)=\theta^{-1}(f(x))$, is a maximal ideal of the ring $vL_{K}%
(E\backslash(H,B_{H}))v$. We wish to show that $N$ is a prime ideal of
$L_{K}(E\backslash(H,B_{H}))$. Let $A,B$ be two ideals of $L_{K}%
(E\backslash(H,B_{H}))$ such that $AB\subset N$. Now $vAvvBv\subset
vABv\subset vNv$ implies that one of them, say $vAv\subset vNv$. We claim that
$A$ does not contain any vertices. Indeed if $A$ contains a vertex $w$, then
$v\in A$ as $w\geq v$. But then $vL_{K}(E\backslash(H,B_{H}))v\subset A$
\ \ and so $vL_{K}(E\backslash(H,B_{H}))v\subset vAv\subset vNv$, a
contradiction to the fact that $vNv$ is a proper ideal of $vL_{K}%
(E\backslash(H,S))v$. Thus $A$ does not contain any vertices and hence, by
Lemma \ref{Lemma 3.4}, the ideal $A$ of $L_{K}(E\backslash(H,B_{H}))$ will be
generated by a polynomial $q(c)$. Since $q(c)=vq(c)v\in vAv\subset vNv\subset
N$, we conclude that $A\subset N$. Thus we have shown that $N$ is a prime
ideal of $L_{K}(E\backslash(H,B_{H}))$. This implies that $P$ is a prime ideal
of $L_{K}(E)$. If $P$ were a graded ideal, then, since $P\cap E^{0}=H$, Lemma
6 implies that $P=I_{(H,B_{H})}$ and this would imply that ($N$ and hence)
$vNv$ must be $0$. But $vNv\overset{\theta}{\cong}<f(x)>\neq0$, a
contradiction. Hence $P$ must be a non-graded prime ideal of $L_{K}(E)$.
\end{proof}

As a consequence of Theorem \ref{Main Theorem}, we get the following Corollary.

\begin{corollary}
\label{Condition K}An arbitrary graph $E$ satisfies the Condition (K) if and
only if every prime ideal of $L_{K}(E)$ is graded.
\end{corollary}

\begin{proof}
Suppose $E$ does not satisfy the Condition (K). Then there is a cycle $c$
without K based at a vertex $v$ in $E$. Define $H=\{w\in E^{0}:w\ngeq v\}$.
Clearly $H$ is a hereditary saturated subset of $E^{0}$. In $E\backslash
(H,B_{H})$, $c$ is then a cycle without exits and based at $v$. Now
$vL_{K}(E\backslash(H,B_{H})v\cong K[x,x^{-1}]$. Choose an irreducible
polynomial $f(x)\in K[x,x^{-1}]$. Since $L_{K}(E)/I_{(H,B_{H})}\overset{\phi
}{\cong}L_{K}(E\backslash(H,B_{H}))$, define an ideal $P$ containing
$I_{(H,B_{H})}$ such that $\phi(P/I_{(H,B_{H})})=<f(c)>$. Then the ideal $P$
being generated by $H\cup\{v^{H}:v\in B_{H}\}\cup\{f(c)\}$ will be a
non-graded prime ideal of $L_{K}(E)$, by Theorem \ref{Main Theorem}.
\ Conversely, it is well-known (see \cite{Goodearl 13}, \cite{Tomforde 20})
that Condition (K) on $E$ implies that every ideal of $L_{K}(E)$ is graded.
\end{proof}

\textbf{Construction of non-graded prime ideals}: The proof of Corollary
\ref{Condition K} provides the following method of constructing non-graded
prime ideals from a cycle without K. Suppose $c$ is a cycle without K based at
a vertex $v$ in a graph $E$. Let $H=\{w\in E^{0}:w\ngeq v\}$. Now $H$ is a
hereditary saturated set and $E^{0}\backslash H=M(v)$. So for each monic
irreducible polynomial $f(x)\in K[x,x^{-1}]$, the ideal $P=<I_{(H,B_{H}%
)},f(c)>$ is, by Theorem \ref{Main Theorem}, is a non-graded prime ideal of
$L_{K}(E)$. From Theorem 3.10, it is clear that $P$ uniquely determines and is
determined by the cycle $c$ and the polynomial $f(x))$.

Using the above construction and Theorem \ref{Main Theorem}, we obtain the
following bijective corrspondence that extends Proposition 3.7 of \cite{Aranda
Pino 9} to arbitrary graphs. It is worth noting that, unlike in \cite{Aranda
Pino 9}, the correspondence involves cycles without K (which are perhaps
easily tractable) instead of maximal tails.

\begin{corollary}
\label{Correspondence Thm}Let $E$ be an arbitrary graph and $K$ be any field.
Then the assignment $P\longmapsto(c,<f(x)>)$ as indicated above and in the
proof of Theorem 3.10 defines a bijection between non-graded prime ideals of
$L_{K}(E)$ and the set $C(E)_{\kappa}\times(Spec(K[x,x^{-1}])\backslash\{0\})$
where $C(E)_{\kappa}$ denotes the set of all cycles without K in $E$ (where
cycles obtained by permuting the vertices of a cycle are considered equal).
\end{corollary}

If we specialize Theorem \ref{Main Theorem} to the case of a row-finite graph
$E$, the set $B_{H}$ must be empty for any hereditary satuated set $H$ and so
Condition (ii) of Theorem \ref{Main Theorem} does not hold. Hence we get the
following extension of Proposition 2.4 of \cite{Aranda Pino 9} which only
deals with graded prime ideals of $L_{K}(E)$ for row-finite graphs $E$.

\begin{corollary}
\label{Row-finite Case} Let $E$ be a row-finite graph and $K$ be any field. An
ideal $P$ \ of $L_{K}(E)$ with $P\cap E^{0}=H$ is a prime ideal if and only if
either $P=<H>$ and $E^{0}\backslash H$ satisfies the MT-3 condition or
$P=<H\cup\{f(c)\}>$, where $c$ is a cycle without K in $E$ based at a vertex
$v$, $H=\{w\in E^{0}:w\ngeq v\}$ and $f(x)$ is an irreducible polynomial in
$K[x,x^{-1}]$.
\end{corollary}

Our investigation of prime ideals enables us to give a simpler proof the
simplicity theorem ( Theeorem 3.11 of \cite{Abrams1}) when $L_{K}(E)$ is a
prime ring.

\begin{proposition}
\label{Simplicity Thm}Let $L_{K}(E)$ be a prime ring. Then $L_{K}(E)$ is a
simple ring if and only if every cycle in $E$ has an exit and the only
hereditary saturated subsets of $E^{0}$ are $E^{0}$ and the empty set $\Phi$.
\end{proposition}

\begin{proof}
Since $L_{K}(E)$ be a prime ring, $E^{0}$ satisfies the MT-3 condition, by
Theorem \ref{primeLPA}.

Suppose $L_{K}(E)$ is a simple ring. Since the ideal generated by a non-empty
proper hereditary saturated subset is a non-zero proper ideal of $L_{K}(E)$,
the simplicity of $L_{K}(E)$ obviously imply that the only hereditary
saturated subsets of $E^{0}$ are $\Phi$ and $E^{0}$. By way of contradiction,
suppose there is a cycle $c$ without exits in $E$. Then since $E^{0}$
satisfies the MT-3 condition, an appeal to Theorem \ref{Main Theorem} yields
that there are infinitely many non-graded prime ideals of $L_{K}(E)$ of the
form $<f(c)>$, for various irreducible polynomials $f(x)$ in $K[x,x^{-1}]$, a
contradiction to the simplicity of $L_{K}(E)$. Thus every cycle in $E$ has an exit.

Conversely, suppose $E^{0}$ satisfies the two conditions. Let $N$ be a proper
ideal of $L_{K}(E)$. Since $E^{0},\Phi$ are the only hereditary saturated
subsets of $E^{0}$, $N\cap E^{0}=\Phi$. Since $E^{0}$ also satisfies the MT-3
condition and every cycle in $E$ has an exit, we appeal to Lemma
\ref{Lemma 3.4} to conclude that $N=0$. Hence $L_{K}(E)$ is simple.
\end{proof}

\textbf{Remark: }In the sequel\textbf{\ }we shall be using the following
observation that follows from Theorem \ref{Main Theorem}: If $P=<I_{(H,B_{H}%
)},f(c)>$ and $Q=<I_{(H,B_{H})},g(c)>$ are two non-graded prime ideals with
$P\cap E^{0}=H=Q\cap E^{0}$, then $P\nsubseteqq Q$ and $Q\nsubseteqq P$. This
is because, if one of the proper inclusion holds, say $P\subsetneqq Q$ then,
from the proof of Theorem \ref{Main Theorem}, $f(c)$ is a divisor of $g(c)$
which will imply that in $K[x,x^{-1}]$, $f(x)$ is a proper divisor of $g(x)$,
contradicting the fact that $g(x)$ is an irreducible polynomial in
$K[x,x^{-1}]$.

\section{Primitive Ideals of $L_{K}(E)$}

We next characterize the primitive ideals of the Leavitt path algebra
$L_{K}(E)$. As noted in \cite{abrams3}, the map $\phi:L_{K}(E)\rightarrow$
$(L_{K}(E))^{op}$ given by \ $%
{\textstyle\sum\limits_{i=1}^{n}}
k_{i}\alpha_{i}\beta_{i}^{\ast}\mapsto%
{\textstyle\sum\limits_{i=1}^{n}}
k_{i}\beta_{i}\alpha_{i}^{\ast}$ is a ring isomorphism and so the distinction
between \ the left and right primitivity of the ring $L_{K}(E)$ vanishes. Thus
a graded ideal $I$ of $L_{K}(E)$ is right primitive if and only if it is left
primitive, since by \cite{Tomforde 20}, $L_{K}(E)/I$ is again a right ( =
left) primitive Leavitt path algebra. That the same conclusion holds for
non-graded ideals will follow from our internal description of primitive
ideals (Theorem \ref{Primitive Ideal Thm} below). Using Theorem
\ref{Primitive Ideal Thm}, we also obtain a description of those graphs $E$
for which every prime ideal of $L_{K}(E)$ is primitive.

We begin with the following important concept introduced in \cite{abrams3}.

\begin{definition}
\label{CSP Definition} Let $E$ be a graph. A subset $S$ of $E^{0}$ is said to
have the Countable Separation Property (CSP, for short), if there is a
countable set $C$ of vertices in $E$ with the property that to each $u\in S$
there is a $v\in C$ such that $u\geq v$.
\end{definition}

For example, if $E$ is a row-finite graph, then for any vertex $v\in E$, the
tree $T(v)=\{w:v\geq w\}$ is a countable set. If in addition, $E^{0}$
satisfies the MT-3 condition, then it is easy to see that $E^{0}$ will always
have the CSP with respect to the tree $T(v)$ for any fixed vertex $v$. \ Also,
in a countable graph $E$, the MT-3 condition will trivially imply the CSP for
$E^{0}$.

The following characterization of primitive Leavitt path algebras of arbitrary
graphs was obtained in \cite{abrams3}.

\begin{theorem}
\label{Primitive LPA}(\cite{abrams3}) Let $E$ be an arbitrary graph and $K$ be
any field. Then the Leavitt path algebra $L_{K}(E)$ is right (= left)
primitive if and only if Condition (L) holds in $E$, and $E^{0}$ satisfies the
MT-3 condition and possesses the countable separation property.
\end{theorem}

By using Theorems \ref{Main Theorem} and \ref{Primitive LPA}, we are able to
characterize the primitive ideals of $L_{K}(E)$ in the next theorem. We also
need the well-known fact (see \cite{Lanski 17}, Theorem 1) that a
not-necessarily unital ring $R$ is right (left) primitive if and only if there
is an idempotent $a\in R$ such that $aRa$ is a right (left) primitive ring.

\begin{theorem}
\label{Primitive Ideal Thm}Let $E$ be an arbitrary graph and $K$ be any field.
Let $P$ be an ideal of $L_{K}(E)$ with $H=P\cap E^{0}$. Then $P$ is right (=
left) primitive if and only if $P$ satisfies one of the following:

(i) $P$ is a non-graded prime ideal;

(ii) $P$ is a graded prime ideal of the form $I_{(H,B_{H}\backslash\{u\})}$
for some $u\in B_{H}$;

(iii) $P$ is a\ graded ideal of the form $I_{(H,B_{H})}$ (where $B_{H}$ may be
empty) and $E^{0}\backslash H$ satisfies the MT-3 condition, the Condition (L)
and the Countable Separation Property.
\end{theorem}

\begin{proof}
Sufficiency: (i) Suppose $P$ is a non-graded prime ideal. We follow the
notation used in the proof of the necessity part of Theorem \ref{Main Theorem}%
. As noted there, the ideal $N\cong P/I_{(H,S)}$ of $L_{K}(E\backslash(H,S))$
is such that $vNv$ is a maximal ideal of $vL_{K}(E\backslash(H,S))v\cong
K[x,x^{-1}]$. Now $v(L_{K}(E\backslash(H,S)v/vNv\cong v^{\prime}%
(L_{K}(E\backslash(H,S))/N)v^{\prime}\cong\bar{v}(L_{K}(E)/P)\bar{v}$ under
the natural isomorphisms, \ where $v^{\prime}=v+N$ and $\bar{v}=v+P$. Since
$v(L_{K}(E\backslash(H,S)v/vNv$\ is a field, $\bar{v}(L_{K}(E)/P)\bar{v}$ is a
commutative primitive ring. As noted above in the statement preceding Theorem
\ref{Primitive Ideal Thm}, $L_{K}(E)/P$ is then a right as well as a left
primitive ring,\ by (Theorem 1, \cite{Lanski 17}). We thus conclude that $P$
is both a right and a left primitive ideal of $L_{K}(E)$.

(ii) Suppose $P$ is a graded prime ideal of the form $I_{(H,B_{H}%
\backslash\{u\})}$ for some breaking vertex $u$ of $H$. By Theorem
\ref{Main Theorem} $E^{0}\backslash H=M(u)$ so that $w\geq u$ for all $w\in
E^{0}\backslash H$. Thus in the graph $E\backslash(H,S)$, $w\geq u^{\prime}$
for all $w\in(E\backslash(H,S))^{0}$ and $u^{\prime}$ is a sink. This shows
that $E\backslash(H,S)$ satisfies not only the MT-3 condition, but also the
countable separation property with respect to $\{u^{\prime}\}$. Moreover,
Condition (L) also holds, since every vertex $w$ on any cycle in the graph
$E\backslash(H,S)$ satisfies $w\geq u^{\prime}$ (and $u^{\prime}$ is a sink).
We appeal to Theorem \ref{Primitive LPA} to conclude that $L_{K}%
(E\backslash(H,S))$ is a right (= left) primitive ring. Since $L_{K}%
(E\backslash(H,S))\cong L_{K}(E)/P$, we then conclude that $P$ is both a right
and a left primitive ideal of $L_{K}(E)$.

(iii) Suppose $P$ is a graded prime ideal of the form $P=I_{(H,B_{H})}$ with
$E^{0}\backslash H$ satisfying the MT-3 condition, the Condition (L) and the
Countable Separation Property. Now by \cite{Tomforde 20}, $L_{K}(E)/P\cong
L_{K}(E\backslash(H,B_{H}))$. From the definition of the graph $E\backslash
(H,B_{H})$, $(E\backslash(H,B_{H}))^{0}=E^{0}\backslash H$ and hence satisfies
the MT-3 condition, the condition (L) and the countable separation property.
By Theorem \ref{Primitive LPA}, $L_{K}(E\backslash(H,B_{H})$ \ and hence
$L_{K}(E)/P$\ is then a right (= left) primitive ring. This shows that $P$ is
both a right and left primitive ideal of $L_{K}(E)$.

Necessity: Follows from the fact that a primitive ideal is always prime and
from Theorem \ref{Primitive LPA} and the cases (i),(ii) of Theorem
\ref{Main Theorem}.
\end{proof}

Theorem \ref{Primitive Ideal Thm} establishes that the distinction between the
right and the left primitivity for ideals in a Leavitt path algebra vanishees.
Hence we shall drop the right/left distinction for primitive ideals of
$L_{K}(E)$. From Theorem \ref{Primitive Ideal Thm}, we can easily describe all
those graphs $E$ for which every prime ideal of $L_{K}(E)$ is primitive.

\begin{corollary}
\label{Prime=Primitive}Let $E$ be an arbitrary graph and $K$ any field. Then
every prime ideal of the Leavitt path algebra $L_{K}(E)$ is primitive if and
only if $E$ satisfies the Condition (K) and every maximal tail $M$ in $E^{0}$
satisfies the Countable separation property.
\end{corollary}

\begin{proof}
Suppose every prime ideal of $L_{K}(E)$ is primitive. First of all there
cannot be any non-graded prime ideals in $L_{K}(E)$. Because a non-graded
prime ideal $P$, from Theorem \ref{Main Theorem}, is of the form
$P=<I_{(H,B_{H})},f(c)>$ and, by Lemma \ref{Lemma 3.6}, it always contains the
prime ideal $Q=I_{_{(H,B_{H})}}$. This leads to a contradiction, since, on the
one hand, the description of $P$ (from Theorem \ref{Main Theorem}) implies
$E^{0}\backslash H$ does not satisfy Condition (L) (as the cycle $c$ has no
exits in $E^{0}\backslash H$) and, on the other hand, the primitivity of $Q$
implies, by Theorem \ref{primeLPA}, that $E^{0}\backslash H$ \ satisfies
Condition (L), a contradiction. Thus every prime ideal of $L_{K}(E)$ is graded
and hence, by Corollary \ref{Condition K}, $E$ satisfies the Condition (K).
For any maximal tail $M$, with $H=E^{0}\backslash M$, if $I_{(H,S)}$ is a
prime ideal (so $S=B_{H}$ or $B_{H}\backslash\{u\}$), then $L_{K}%
(E)/I_{(H,S)}\cong L_{K}(E\backslash(H,S))$ is a primitive ring and so, by
Theorem \ref{Primitive LPA}, $(E\backslash(H,S))^{0}=E^{0}\backslash H=M$
satisfies the Condition (L) and the countable separation property.

Conversely, it is easy to see that the Condition (K) in $E$ implies that any
maximal tail in $E^{0}$ also satisfies the Condition (K) and hence the
Condition (L). Moreover, every prime ideal of $L_{K}(E)$ is graded. For any
graded prime ideal $I_{(H,S)}$, $(E\backslash(H,S))^{0}=E^{0}\backslash H$ is
a maximal tail that satisfies, by hypothesis, the Condition (L) and the
countable separation property. So $L_{K}(E)/I_{(H,S)}\cong L_{K}%
(E\backslash(H,S))$ is a primitive ring by Theorem \ref{Primitive LPA}. This
implies that $I_{(H,S)}$ is a primitive ideal.
\end{proof}

Here are a few graphs $E$ satisfying the conditions of Theorem
\ref{Prime=Primitive} (and thus every prime ideal of the corresponding
$L_{K}(E)$ is primitive): (a) $E$ is any countable acyclic graph; (b) $E$ is a
row-finite graph satisfying the condition (L); (c) $E$ is a graph consisting
of a straight line graph \ $\overset{v_{1}}{\bullet}\rightarrow\overset{v_{2}%
}{\bullet}\rightarrow\overset{v_{3}}{\bullet}\rightarrow\bullet\rightarrow
\bullet\cdot\cdot\cdot\cdot$ \ together with a cycle $%
\begin{array}
[c]{ccc}%
u & \overset{\leftarrow}{\rightarrow} & w
\end{array}
$ where, for each $n$, there is an edge from $u$ to $v_{n}$. Thus $u$ is a
breaking vertex for $H=\{v_{1},v_{2},.....\}$. Here $M=\{u,w\}$ is the only
maximal tail in $E$ and it satisfies condition (i) of Theorem
\ref{Prime=Primitive}.

\subsection{Examples}

We now construct various examples of prime and primitive ideals of Leavitt
path algebras illustrating the various possibilities mentioned in Theorems
\ref{Main Theorem} and \ref{Primitive Ideal Thm}.

\begin{example}
\label{Ex 4.5}Consider the graph $E$ where $E^{0}=\{v_{i},u_{i}%
:i=1,2,3,.....\}$. For each $i$, there is exactly one edge from $u_{i}$ to
$u_{i+1}$. For each $j$, EACH $v_{j}$ is an infinite emitter connected to
\ EACH $u_{i}$ by an edge. Also each $v_{j}$ connects to $v_{j+1}$ by an edge
and vice versa (to form a cycle of length 2). Thus $E$ looks like%
\[%
{\includegraphics[
natheight=1.069800in,
natwidth=3.153100in,
height=1.1035in,
width=3.1981in
]%
{LN92WS00.bmp}%
}
\]

\end{example}

It is easy to check that there are only three hereditary saturated subsets of
vertices in the graph $E$, namely, $H=\{u_{1},u_{2},.....\}$, $E^{0}$ and the
empty set. Now $B_{H}=\{v_{1,}v_{2},....\}$. For each $j=1,2,...$ , let
$S_{j}=B_{H}\backslash\{v_{j}\}$. Observe that in $E^{0}\backslash H$,
$v_{i}\geq v_{j}$ and $v_{j}\geq v_{i}$ for any two $i$ and $j$. It is then
clear that $E^{0}$, $E^{0}\backslash H$ and, for each $j$, $(E\backslash
(H,S_{j})^{0}=(E^{0}\backslash H)\cup\{v_{j}^{\prime}\}$ all satisfy the MT-3
condition. As an application of Theorem \ref{Main Theorem}, we then get, for
each $j$, the ideal $P_{j}=I_{(H,S_{j})}$, the ideal $P=I_{(H,B_{H})}$ and
$\{0\}$ are all graded prime ideals of $L_{K}(E)$. Note that $E$ satisfies the
Condition (K). Further $E^{0}$, being countable, has the countable separation
property. We thus conclude the following: (i) $\{0\}$, $P$, $P_{j}$ ($j\geq1$)
are the only prime ideals of $L_{K}(E)$; (ii) Every prime ideal of $L_{K}(E)$
is graded and (iii) All the prime ideals of $L_{K}(E)$ are primitive (by
Theorem \ref{Prime=Primitive}).

Remark: In Example \ref{Ex 4.5}, every prime ideal\ of $L_{K}(E)$ \ is graded.
A natural question is: Is there a graph $E$ such that every prime ideal in
$L_{K}(E)$ is non-graded ? The answer is in the negative. Indeed, as shown in
Lemma \ref{Lemma 3.6} and in view of Theorem \ref{Main Theorem}, if $J$ is a
non-graded prime ideal of $L_{K}(E)$ with $\ H=J\cap E^{0}$, then the ideal
$I_{(H,B_{H})}$ is always a graded prime ideal of $L_{K}(E)$.

\begin{example}
\label{Ex 4.6}Consider the graph $F$ given below.%
\[%
{\includegraphics[
natheight=1.208100in,
natwidth=3.416900in,
height=1.2427in,
width=3.4627in
]%
{LN92WS01.bmp}%
}
\]

\end{example}

Here $v$ is an infinite emitter, connected to every vertex $u_{j}$. Let $c$
denote the loop at $v$. Now $H=\{u_{1},u_{2},....\}$ is the only proper
non-empty hereditary saturated subset of $E^{0}$ with $B_{H}=\{v\}$ and $c$
has no exits in $E^{0}\backslash H$. Let $J$ be the two-sided \ ideal
generated by $H\cup\{v-cc^{\ast}\}\cup$ $\{v-c\}$. By Theorem
\ref{Main Theorem}, $J$ is a non-graded prime ideal of $L_{K}(F)$ and is thus
primitive. Note that the non-graded prime (= primitive) ideal of $L_{K}(F)$
are in one-to-one correspondence with the irreducible polynomials in
$K[x.x^{-1}]$. \ By Theorem \ref{Main Theorem}, the ideal $Q=I_{(H,B_{H}%
)}=<H,v-cc^{\ast}>$ is a graded prime ideal of $L_{K}(E)$ and is not primitive
as $E^{0}\backslash H$ does not satisfy Condition (L). Also $\{0\}$ is a
(graded) prime\ (actually primitive) ideal of $L_{K}(E)$ as $E^{0}$ satisfies
the MT-3 condition, the Condition (L) and the CSP. The poset $Spec(L_{K}(F))$
under set inclusion can be pictorially described as $%
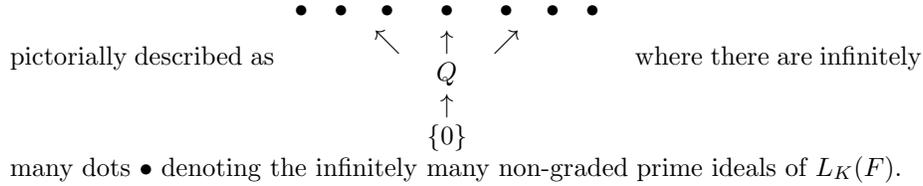
\begin{array}
[c]{ccccccc}
&  &  &  &  &  & \\
\bullet & \bullet & \bullet & \bullet & \bullet & \bullet & \bullet\\
&  & \nwarrow & \uparrow & \nearrow &  & \\
&  &  & Q &  &  & \\
&  &  & \uparrow &  &  & \\
&  &  & \{0\} &  &  &
\end{array}
$ $\ $\ where there are infinitely many dots $\bullet$ denoting the infinitely
many non-graded prime ideals of $L_{K}(F)$.

\begin{example}
\label{Ex 4.7}Let $G$ be the graph given below.
\[%
{\includegraphics[
natheight=1.916400in,
natwidth=3.277600in,
height=1.9545in,
width=3.3226in
]%
{LN92WS02.bmp}%
}
\]

\end{example}

This graph $G$ is essentially the graph $E$ of \ref{Ex 4.5} except for an
additional vertex $w$, an edge from $v_{1}$ to $w$ and a loop at $w$ denoted
by $c$. In addition to $G^{0}$ and the empty set, there are two hereditary
saturated subsets in $G^{0}$, namely, $H_{1}=\{u_{1},u_{2},.....\}$ and
$H_{2}=\{u_{1},u_{2},.....\}\cup\{w\}$. Note that $B_{H_{1}}=B_{H_{2}}%
=\{v_{1},v_{2},....\}$. $L_{K}(G)$ has infinitely many graded prime ideals
and\ also infinitely many non-graded prime ideals. For example, $G^{0}%
\backslash H_{1}=M(w)$ satisfies the MT-3 condition and has a cycle/loop $c$
without exits based at $w$. For each irreducible polynomial $f(x)\in
K[x,x^{\_1}]$, the ideal $P_{f(x)}$ generated by $H_{1}\cup\{v^{H_{1}}:v\in
B_{H_{1}}\}\cup\{f(c)\}$ is, by \ref{Main Theorem}, a non-graded prime ideal
of $L_{K}(G)$. Since $G^{0}\backslash H_{2}=M(v_{i})$ for each $v_{i}$, the
ideal $I_{(H_{2,}S_{i})}$ is a graded prime ideal of $L_{K}(G)$ where
$S_{i}=B_{H_{2}}\backslash\{v_{i}\}$. Note that the graded ideals
$I_{(H_{2,}S_{i})}$ are all primitive. On the other hand , the graded ideal
$I_{(H_{1},B_{H_{1}})}$ is a prime ideal which is not primitive.

\section{\textbf{The} \textbf{Stratification of }$Spec(L_{K}(E)$}

In the case of an arbitrary graph $E$, the stratification of the prime
spectrum of $L_{K}(E)$ is interestingly similar, but different from the
stratification of $Spec(L_{K}(E)$ for a row-finite graph $E$ (\cite{Abrams2}).
For each cycle $c$ without K and based at a vertex $v$ in the arbitrary graph
$E$, let $M_{c}=\{w\in E^{0}:w\geq v\}$ and $H_{c}=\{w\in E^{0}:w\ngeq v\}$.
\ Note that $M_{c}$ is the smallest maximal tail containing $c$. For each
maximal tail $M$ in $E^{0}$, define the \textit{stratum corresponding to }$M$
with $E^{0}\backslash M=H$ to be
\[
Spec_{M}(L_{K}(E))=\{P\in Spec(L_{K}(E)):P\cap E^{0}=H\}.
\]
Thus the prime spectrum of the Leavitt path algebra $L_{K}(E)$ is the union of
disjoint strata corresponding to\ distinct maximal tails.

For a given cycle $c$ without K and based at a vertex $v$ with $M_{c}=\{w\in
E^{0}:w\geq v\}$, it is clear from the proof of Theorem \ref{Main Theorem}
that, the corresponding stratum $Spec_{M_{c}}(L_{K}(E))$ consists of at most
two graded prime ideals $P_{h}=I_{(H,B_{H})}$ and $I_{(H,B_{H}\backslash
\{v\})}$ (in case $v\in B_{H}$), where $H=\{w\in E^{0}:w\ngeq v\}$, together
with an infinite set of non-graded prime ideals all containing the ideal
$I_{(H,B_{H}})$ and indexed by the irreducible polynomials in $K[x,x^{-1}%
]$.\ Thus if $v\notin B_{H}$, $P_{h}$ will be the only graded prime ideal of
$L_{K}(E)$ and in this case, $Spec_{M_{c}}(L_{K}(E))$ is homeomorphic to
$Spec(K[x,x^{-1}])$ with $P_{h}$ corresponding to the ideal $\{0\}$ of
$K[x,x^{-1}]$. The poset of $Spec_{M_{c}}(L_{K}(E))$ under set inclusion will
look like%
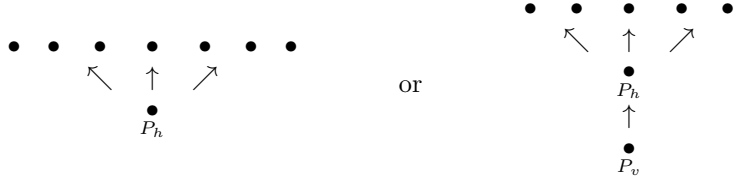
\[%
\begin{array}
[c]{ccccccc}%
\bullet & \bullet & \bullet & \bullet & \bullet & \bullet & \bullet\\
&  & \nwarrow & \uparrow & \nearrow &  & \\
&  &  & \underset{P_{h}}{\bullet} &  &  &
\end{array}
\text{ \quad\qquad or \qquad\quad}%
\begin{array}
[c]{ccccc}%
\bullet & \bullet & \bullet & \bullet & \bullet\\
& \nwarrow & \uparrow & \nearrow & \\
&  & \underset{P_{h}}{\bullet} &  & \\
&  & \uparrow &  & \\
&  & \underset{P_{v}}{\bullet} &  &
\end{array}
\]
where the dots $\bullet$ denote infinitely many non-graded prime ideals of
$L_{K}(E)$. Observe that in the above case, the cycle $c$ has no exits in
$M_{c}$.

If $M$ is a maximal tail in which every cycle has an exit, it determines at
most two graded prime ideals $P_{u}=I_{(H,B_{H}\backslash\{u\})}$ (when there
is a $u\in B_{H}$ such that $H=\{w\in E^{0}:w\ngeq u\}$) and $P_{h}%
=I_{(H,B_{H})}$ (with $P_{h}$ being the only graded prime ideal, if
$B_{H}=\Phi$ or if $E^{0}\backslash H\neq M(u)$ for any $u\in B_{H}$). Note
that $P_{u}\subset P_{h}$ and that $P_{u}$ will always be primitive by Theorem
\ref{Primitive Ideal Thm}. Thus the\ poset of $Spec_{M}(L_{K}(E))$ under set
inclusion will look like $%
\begin{array}
[c]{cc}%
\overset{P_{h}}{\bullet} & \\
\uparrow & \\
\underset{P_{u}}{\bullet} &
\end{array}
$ or\quad\ $\overset{P_{h}}{\bullet}$.

\section{Leavitt path algebras with prescribed Krull Dimension}

Recall that a ring $R$ is said to have \textbf{Krull dimension }$n$, if $n$ is
the supremum of all non-negative integers $k$\ such that there is a chain of
prime ideals $P_{0}\subsetneqq\cdot\cdot\cdot\subsetneqq P_{k}$. We can extend
this definition to infinite ordinals $\lambda$. Thus $R$ is said to have Krull
dimension $\lambda$, if $\lambda$ is the supremum of the order types of all
the continuous well-ordered of ascending chains of prime ideals in $R$.

In this section, we characterize the Leavitt path algebras with Krull
dimension $0$. We also construct examples of Leavitt path algebras with
various prescribed Krull dimensions (both finite and infinite).

We begin by describing the graphical conditions on $E$ under which every
non-zero prime ideal of $L_{K}(E)$ is maximal.

Recall that every non-empty subset $X$ of vertices in a graph $E$ gives rise
to the restricted subgraph $E_{X}$ where $(E_{X})^{0}=X$ and $(E_{X}%
)^{1}=\{e\in E^{1}:s(e),r(e)\in X\}$.

\begin{theorem}
\label{Prime=maximal}Let $E$ be an arbitrary graph and $K$ be any field. Then
every non-zero prime ideal of the Leavitt path algebra $L_{K}(E)$ is maximal
if and only if $E$ satisfies one of the following two conditions:

Condition \ I: (i) $E^{0}$ is a maximal tail; (ii) The only hereditary
saturated subsets of $E^{0}$ are $E^{0}$ and $\Phi$ (the empty set); (iii) $E$
does not satisfy the Condition (K).

Condition II: (a) $E$ satisfies the Condition (K); (b) For each maximal tail
$M$, the restricted graph $E_{M}$ contains no proper non-empty hereditary
saturated subsets; (c) If $H$ is a hereditary saturated subset of $E^{0}$,
then for each $u\in B_{H}$, $M(u)\subsetneqq E^{0}\backslash H$.
\end{theorem}

\begin{proof}
Suppose every non-zero prime ideal of $L_{K}(E)$ is a maximal ideal. We
distinguish two cases.

Case 1: Suppose there exists a non-graded prime ideal $P$ in $L_{K}(E)$. We
shall show that $E$ satisfies Condition I. From Theorem \ref{Main Theorem},
$P=<I_{(H,B_{H})},f(c)>$, where $H=P\cap E^{0}$ and $c$ is a cycle without K
based at a vertex $v$ in $E$ and that $E^{0}\backslash H=M(v)$. Since, by
Lemma \ref{Lemma 3.6}, $P$ contains the graded prime ideal $I_{(H,B_{H})}$
which, if non-zero, must be a maximal ideal, we conclude that $H=\Phi$. So
$P=<f(c)>$ and contains no vertices. Now $E^{0}\backslash H=E^{0}$ and $w\geq
v$ for every $w\in E^{0}$. The last property implies (as $c$ is a cycle
without K) that $c$ has no exits in $E^{0}$ (thus, in particular, $E$ does not
satisfy the Condition (K) ) and that $E^{0}$ is a maximal tail.\ It also
implies that any non-empty hereditary saturated subset $X$ of $E^{0}$ contains
$c^{0}$. So if $Q=<X>$, then $Q$ contains $c$ and hence properly contains
$<f(c)>=P$. By maximality of $P$, $Q=L_{K}(E)$ and so $X=E^{0}$. \ Thus,
$\Phi$ and $E^{0}$ are the only hereditary saturated subsets of $E^{0}$. ( In
particular, all the non-zero ideals of $L_{K}(E)$ are non-graded). This proves
that $E$ satisfies Condition I. (Note also that, in this case, $L_{K}(E)$ is a
prime ring, by Theorem \ref{primeLPA}) \ 

Case 2: Suppose every prime ideal of $L_{K}(E)$ is graded. So, by Corollary
\ref{Condition K}, the graph $E$ satisfies the Condition (K).\ We wish to
establish condition II (b),(c) for $E$. Now condition II(c) must hold, since
otherwise there will be a hereditary saturated subset $H$ of $E^{0}$and a
$u\in B_{H}$ such that $M(u)=E^{0}\backslash H$. Then, by Theorem
\ref{Main Theorem}, the ideal $I_{(H,B_{H}\backslash\{u\})}$ will be a graded
prime ideal and, since it is properly contained in $I_{(H,B_{H})}$, we get a
contradiction to the maximality of $I_{(H,B_{H}\backslash\{u\})}$. To prove
condition II(b), let $M$\ be any maximal tail with $E^{0}\backslash M=H$. By
Theorem \ref{Main Theorem}(i), $I_{(H,B_{H})}$ is a prime ideal and hence by
supposition, a maximal ideal of $L_{K}(E)$. Then $L_{K}(E\backslash
(H,B_{H}))\cong L_{K}(E)/I_{(H,B_{H})}$ is a simple ring and so, by Theorem
3.11 of \cite{Abrams1}, $M=E^{0}\backslash H=(E\backslash(H,S))^{0}$ contains
no proper non-empty hereditary saturated subsets. Thus $E$ satisfies Condition II.

Conversely, suppose $E$ satisfies condition I of the Theorem. Since Condition
(K) does not hold, $L_{K}(E)$ cannot be a simple ring by Lemma 4.1 of
\cite{AA-2} and so there are non-zero ideals in\ $L_{K}(E)$. By hypothesis
I(ii) all the non-zero ideals of $L_{K}(E)$ are non-graded. Moreover,
hypotheses I(i), (ii), together with Lemma \ref{Lemma 3.4}, imply that there
exists a cycle $c$ without exits in $E^{0}$ and that every non-zero ideal $J$
of $L_{K}(E)$ is of the form $<g(c)>$ for some polynomial in $g(x)\in
K[x]\subset$ $K[x,x^{-1}]$. Let $P$ be a non-zero prime ideal of $L_{K}(E)$.
Since $P$ is non-graded, by Theorem 3.10, $P=<p(c)>$ where $p(x)$ is an
irreducible polynomial in $K[x,x^{-1}]$. If $J=<g(c)>$ is an ideal with
$J\nsubseteqq P$,\ then $g(x)\notin<p(x)>$ and by maximality,
$<g(x)>+<p(x)>=K[x,x^{-1}]$. Since $vL_{K}(E)v\cong K[x,x^{-1}]$, we conclude
that $v\in vJv+vPv\subset J+P$. Thus $(J+P)\cap E^{0}\neq\Phi$ and so, by
hypothesis I(ii), $E^{0}\subset J+P$. Hence $J+P=L_{K}(E)$. This proves that
$P$ is a maximal ideal of $L_{K}(E)$.

Suppose now $E$ satisfies the Condition II. Now the Condition (K) implies that
every ideal of $L_{K}(E)$ is graded. By Condition II(c) and Theorem
\ref{Main Theorem} (i), (ii), every non-zero prime ideal $P$ of $L_{K}(E)$ is
of the form $I_{(H,B_{H})}$ where $H=P\cap E^{0}$. Let $M=E^{0}\backslash H$.
Now the Condition (K) in $E$ implies that $E\backslash(H,B_{H})$ satisfies the
Condition (K) and, since $(E\backslash(H,B_{H}))^{0}=E^{0}\backslash H=M$, our
hypothesis implies that $E\backslash(H,B_{H})=E_{M}$ contains no non-empty
proper hereditary saturated subsets of vertices and so, by Theorem 3.11 of
\cite{Abrams1}, $L_{K}(E)/I_{(H,B_{H})}\cong L_{K}(E\backslash(H,B_{H}))$ is a
simple ring, thus proving that $I_{(H,B_{H})}$ is a maximal ideal of
$L_{K}(E)$.
\end{proof}

When every non-zero prime ideal of $L_{K}(E)$ is maximal and $L_{K}(E)$
contains a non-graded prime ideal, then \ our proof of \ Case I in Theorem
\ref{Prime=maximal} shows that there is a unique cycle $c$ based at a vertex
$v$ and that $M(v)=E^{0}$. Gene Abrams points out that if further, $E$ is a
finite graph, then in this case, $L_{K}(E)\cong M_{n}(K[x,x^{-1}])$ where $n$
is the number of paths in $E$ which end in $c$ but do not contain $c$. This
was shown in Theorem 3.3 of ( \cite{Abrams6}). On the other hand, since
$K[x,x^{-1}]$ is a Euclidean domain, every non-zero prime ideal of
$K[x,x^{-1}]$ and hence also of $M_{n}(K[x,x^{-1}])$ is maximal. Thus we get
the following corollary.

\begin{corollary}
\label{Finite case Prime=maximal}Let $E$ be a finite graph. Then every
non-zero prime ideal of $L_{K}(E)$ is maximal if and only if \ either
$L_{K}(E)\cong M_{n}(K[x,x^{-1}])$ for some positive integer $n$ or $E$
satisfies the Condition (K) and for each maximal tail $M$, the restricted
graph $E_{M}$ contains no proper non-empty hereditary saturated subsets of vertices.
\end{corollary}

Next we shall use Theorem \ref{Prime=maximal} and Proposition
\ref{Simplicity Thm} to describe those Leavitt path algebras whose Krull
dimension is zero.

Recall that a subset $S$ of a partially ordered set $(X,\leq)$ is called an
\textbf{antichain }if for any two $a,b\in S$, we have $a\nleqslant b$ and
$b\nleqslant a$.

We begin with some preliminary observations. Clearly, a maximal ideal $M$ of a
ring $R$ with multiplicative identity $1$ is a prime ideal of $R$. If $R$ is a
ring without identity but $R^{2}=R$ (such as a Leavitt path algebra), then
also a maximal ideal $M$ of $R$ is prime. This is because, if $A$, $B$ are two
ideals of $R$ such that $A\nsubseteq M$ and $B\nsubseteq M$ so that $A+M=R$
and $B+M=R$, then $AB\nsubseteq M$ since otherwise $R=R^{2}%
=(A+M)(B+M)=AB+AM+MB+MM\subset M$, a contradiction.

If $R$ is a ring with identity, then, by Zorn's Lemma, every (prime) ideal of
$R$ embeds in a maximal ideal and so the Krull dimension of $R$ is $0$ if and
only if every prime ideal of $R$ is maximal. But this no longer is true in a
Leavitt path algebra $L_{K}(E)$ if $E$ is an arbitrary graph. Because, as the
following example shows, maximal ideals need not exist in $L_{K}(E)$. (In
contrast, it is worth noting that maximal one sided ideals always exist in any
$L_{K}(E)$: For a vertex $v$ in $E$, apply Zorn's Lemma pick a maximal left
$L_{K}(E)$-submodule $M$ of $L_{K}(E)v$. Then $M\oplus C$ is a desired maximal
left ideal of $L_{K}(E)$, where $L_{K}(E)=L_{K}(E)v\oplus C$).

\textbf{Example}: Let $E$ be the graph consisting of an infinite line segement
in which there are two loops at each vertex $v_{i}$ and, for each $i$, there
is an edge from vertex $v_{i+1}$ to vertex $v_{i}$.Thus $E$ looks like:%

\begin{figure}
[h]
\begin{center}
\includegraphics[
natheight=0.861400in,
natwidth=3.653000in,
height=0.8934in,
width=3.7005in
]%
{LN92WS03.bmp}%
\end{center}
\end{figure}
Since $E$ satisfies the Condition (K), every ideal in $L_{K}(E)$ is graded and
is generated by vertices (since $E$ is row-finite) \ Now the proper non-empty
hereditary saturated subsets of $E^{0}$ ae precisely the sets $H_{n}%
=\{v_{1},...,v_{n}\}$ for various positive integers $n$. So the non-zero
ideals of $L_{K}(E)$ are of the form $<H_{n}>$, all of which are actually
prime ideals of $L_{K}(E)$ (since $E^{0}\backslash H_{n}$ is a maximal tail
for each $n$). These ideals form an ascending chain $\{0\}\subset
H_{1}>\subset...\subset<H_{n}>\subset....$ . Obviously $L_{K}(E)$ has no
maximal ideals.

\begin{theorem}
\label{KrullDim0}Let $E$ be an arbitrary graph and $K$ be any field. Then the
Leavitt path algebra $L_{K}(E)$ has Krull dimension $0$ if and only if\ $E$
satisfies the Condition (K) and EITHER $E^{0}$ is a maximal tail and contains
no non-empty proper maximal tails\ OR the maximal tails in $E^{0}$ are proper
subsets of $E^{0}$ and form an antichain under set inclusion and\ no maximal
tail $M$ is of the form $M=M(u)$ for any $u\in B_{H}$ where $H=E^{0}\backslash
H$.
\end{theorem}

\begin{proof}
Suppose $L_{K}(E)$ has Krull dimension $0$. This means that if $P$ is a prime
ideal of $L_{K}(E)$, then there cannot be another prime ideal $Q$ with
$P\subsetneqq Q\subsetneqq L_{K}(E)$. Suppose $\{0\}$ is a prime ideal of
$L_{K}(E)$. Then, by Theorem \ref{primeLPA}, $E^{0}$ is a maximal tail and
$L_{K}(E)$ has no other prime ideals. So there are no non-empty proper maximal
tails. Since, in particular, there are no non-graded prime ideals in
$L_{K}(E)$, Corollary \ref{Condition K} implies that $E$ satisfies the
Condition (K).

Suppose now that $\{0\}$ is not a prime ideal of $L_{K}(E)$. So $E^{0}$ is not
a maximal tail. We claim that there cannot be a non-graded prime ideal $P$ in
$L_{K}(E)$. Because, by Theorem \ref{Main Theorem}, such a $P$ will be of the
form $P=<I_{(H,B_{H})},f(c)>\supsetneqq I_{(H,B_{H})}$ which will also (even
if it is $0$ ) be a prime ideal by Lemma \ref{Lemma 3.6}, thus contradicting
the fact that $L_{K}(E)$ has Krull dimension $0$. Thus every prime ideal of
$L_{K}(E)$ must be graded and so, again by Corollary \ref{Condition K}, $E$
satisfies the Condition (K)$.$ Also if $M_{1}$, $M_{2}$ are two distinct
maximal tails in $E^{0}$, they must be proper subsets of $E^{0}$ and if one is
contained in the other, say, $M_{1}\subseteq M_{2}$ with $H_{1}=E^{0}%
\backslash M_{1}$ and $H_{2}=E^{0}\backslash M_{2}$, then we get the prime
ideals $I_{(H_{2},B_{H_{2}})}\subsetneqq I_{(H_{1}B_{H_{1})}}\neq L_{K}(E)$,
contradicting the fact that $L_{K}(E)$ has Krull dimension $0$. Thus the
maximal tails in $E^{0}$ form an antichain under set inclusion. If a maximal
tail $M=M(u)$ for some $u\in B_{H}$ where $H=E^{0}\backslash M$, then, by
Theorem \ref{Main Theorem}, we get the prime ideals $I_{(H,B_{H})}\supsetneqq
I_{(H,B_{H}\backslash\{u\}}$, again contradicting that $L_{K}(E)$ has Krull
dimension $0$. This proves the necessity .

Conversely, if $E^{0}$ is a maximal tail, then, by Theorem \ref{primeLPA},
$\{0\}$ is a prime ideal of $L_{K}(E)$ and if $E^{0}$contains no non-empty
proper maximal tails, then $\{0\}$ will be the only prime ideal of $L_{K}(E)$
and in this case $L_{K}(E)$ will clearly have Krull dimension $0$. Next
suppose $E$ satisfies the three stated conditions. Now Condition (K) and
Theorem \ref{Main Theorem} implies that every prime ideal in $L_{K}(E)$ must
be of the form $I_{(H,S)}$ with $E^{0}\backslash H=M$ a maximal tail, where
$S=B_{H}$ or $B_{H}\backslash\{u\}$ for some $u\in B_{H}$. Note the latter
case is not possible, since, in that case, $M=M(u)$, contradicting the
hypothesis. Thus all the prime ideals of $L_{K}(E)$ are non-zero and must be
of the form $I_{(H,B_{H})}$. Let $P_{1}=I_{(H_{1},B_{H_{1}})}$ be a prime
ideal of $L_{K}(E)$. If there is another prime ideal $Q=I_{(H_{2},B_{H_{2}})}$
such that $P\subsetneqq Q$, then we get maximal tails $E^{0}\backslash
H_{2}\subsetneqq E^{0}\backslash H_{1}\neq E^{0}$, contradicting the fact that
the maximal tails in $E^{0}$ form an antichain under set inclusion. Thus
$L_{K}(E)$ has Krull dimension $0$.
\end{proof}

When $E$ is a finite graph, Theorem \ref{KrullDim0} reduces to the following corollary.

\begin{corollary}
\label{Finite Krull Dim 0}Let $E$ be a finite graph. Then the Leavitt path
algebra $L_{K}(E)$ has Krull dimension zero if and only if either $L_{K}(E)$
is a prime simple ring or $E^{0}$ satisfies the Condition (K), every maximal
tail $M$ in $E^{0}$ satisfies $M\neq E^{0}$and $E_{M}$ has no non-empty proper
hereditary saturated subset of vertices.
\end{corollary}

\begin{proof}
Now $L_{K}(E)$ has a multiplicative idntity and so it will have Krull
dimension zero if and only if every prime ideal of $L_{K}(E)$ is maximal.
\ Suppose $L_{K}(E)$ has Krull dimension zero. If $\{0\}$ is a prime ideal
\ and hence a maximal ideal of $L_{K}(E)$, then $L_{K}(E)$ becomes a prime
ring which is simple. Suppose $\{0\}$ is not a prime ideal. Then $E^{0}$ does
not satisfy the MT-3 condition. As shown in the second paragraph of the proof
ofTheorem \ref{KrullDim0}, it is then clear that $E$ satisfies the Condition
(K) so that, in particular, every prime ideal is a graded ideal of the form
$I=<H>$, where $H=I\cap E^{0}$ with $M=E^{0}\backslash H$ a maximal tail.
Since $L_{K}(E\backslash(H)\cong L_{K}(E)/I$ \ is now a simple ring, Theorem
3.11 of \cite{Abrams1} implies that $E_{M}$, where $M=E^{0}\backslash
H=(E\backslash H)^{0}$, contains no proper non-empty hereditary saturated subsets.

Conversely, a prime simple ring clearly has Krull dimension zero. Suppose now
that $E$ satisfies the second set of stated conditions. Condition (K) implies
that every ideal of $L_{K}(E)$ is graded and so every prime ideal $P$ of
$L_{K}(E)$ is of the form $P=<H>$ where $H=P\cap E^{0}$ with $M=E^{0}%
\backslash H$ a maximal tail. By hypothesis, $(E\backslash H)^{0}=(E_{M})^{0}$
has no proper non-empty hereditary saturated subsets and $E\backslash H$
satisfies the Condition (K). So, by Theorem 3.11 of \cite{Abrams1},
$L_{K}(E)/P\cong L_{K}(E\backslash(H)$\ is a simple ring showing that $P$ is a
maximal ideal of $L_{K}(E)$. Hence $L_{K}(E)$ has Krull dimension zero.
\end{proof}

\subsection{Examples.}

The following examples illustrate graphs satisfying the properties stated in
Theorems \ref{Prime=maximal} and \ref{KrullDim0}, and Corollary
\ref{Finite Krull Dim 0}.

\begin{example}
\label{Ex 6.5}Let $E_{1}$ be the graph

$%
\begin{array}
[c]{ccccccccccc}
&  & \overset{w_{1}}{\bullet} &  &  &  & \overset{u_{1}}{\bullet} & \leftarrow
& \overset{v_{1}}{\bullet} &  & \\
& \swarrow &  & \searrow &  & \swarrow &  &  &  & \searrow & \\
\overset{w}{\bullet} &  &  &  & \overset{u_{3}}{\bullet} &  &  &  &  &  &
\overset{v}{\bullet}\\
& \nwarrow &  & \nearrow &  & \nwarrow &  &  &  & \nearrow & \\
&  & \overset{w_{2}}{\bullet} &  &  &  & \overset{u_{2}}{\bullet} & \leftarrow
& \overset{v_{2}}{\bullet} &  &
\end{array}
$
\end{example}

The graph $E_{1}$ is acyclic (hence trivially satisfies Condition (K)),
contains three non-trivial hereditary saturated subsets
\[
H=\{u_{1},u_{2},u_{3}\},H_{1}=\{w,w_{1},w_{2},u_{1},u_{2},u_{3}\},H_{2}%
=\{u_{1},u_{2},u_{3},v_{1},v_{2},v\}
\]
and two maximal tails $M_{1}=E_{1}^{0}\backslash H_{1}=\{v_{1},v_{2},v\}$
\ and $M_{2}=E_{1}^{0}\backslash H_{2}=\{w,w_{1},w_{2}\}$. Clearly, $M_{1}$
and $M_{2}$ contain no non-empty proper hereditary saturated subsets of
vertices. Thus $E_{1}$ satisfies the second condition of Corollary
\ref{Finite Krull Dim 0}. The prime (= maximal) ideals of $L_{K}(E_{1})$ are
$<H_{1}>$ and $<H_{2}>$. Thus $L_{K}(E_{1})$ has Krull dimension $0$. Note
that $<H>$ is not a prime ideal as $E_{1}^{0}\backslash H$ is not a maximal tail.

\begin{example}
\label{Ex 6.6}Let $E_{2}$ be the graph

$%
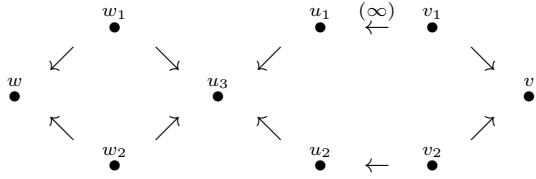
\begin{array}
[c]{ccccccccccc}
&  & \overset{w_{1}}{\bullet} &  &  &  & \overset{u_{1}}{\bullet} &
\overset{(\infty)}{\leftarrow} & \overset{v_{1}}{\bullet} &  & \\
& \swarrow &  & \searrow &  & \swarrow &  &  &  & \searrow & \\
\overset{w}{\bullet} &  &  &  & \overset{u_{3}}{\bullet} &  &  &  &  &  &
\overset{v}{\bullet}\\
& \nwarrow &  & \nearrow &  & \nwarrow &  &  &  & \nearrow & \\
&  & \overset{w_{2}}{\bullet} &  &  &  & \overset{u_{2}}{\bullet} & \leftarrow
& \overset{v_{2}}{\bullet} &  &
\end{array}
$
\end{example}

Now $E_{2}$ is the same as the graph $E_{1}$ with the exception that $v_{1}$
is now an infinite emitter with $r(s^{-1}(v_{1}))=\{u_{1}\}$ as indicated by $%
\begin{array}
[c]{ccc}%
\overset{u_{1}}{\bullet} & \overset{(\infty)}{\leftarrow} & \overset{v_{1}%
}{\bullet}%
\end{array}
$. Here $B_{H_{1}}=\{v_{1}\}$. Since $v\ngeqq v_{1}$, condition II(c) of
Theorem \ref{Prime=maximal}\ is satisfied. As before the conditions II(a),(b)
are also satisfied by $E_{2}$. The prime ideals of $L_{K}(E_{2})$ are
$I_{(H_{1},B_{H_{1}})}$ and $I_{(H_{2},\Phi)}$. Thus $L_{K}(E_{2})$ has Krull
dimension $0$.

\begin{example}
\label{Ex 6.7}Let $E_{3}$ be the graph

$%
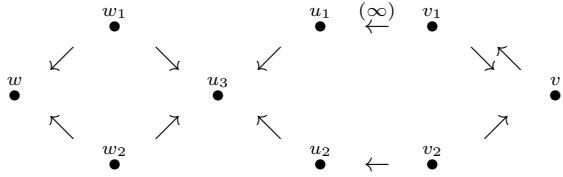
\begin{array}
[c]{ccccccccccc}
&  & \overset{w_{1}}{\bullet} &  &  &  & \overset{u_{1}}{\bullet} &
\overset{(\infty)}{\leftarrow} & \overset{v_{1}}{\bullet} &  & \\
& \swarrow &  & \searrow &  & \swarrow &  &  &  & \searrow\nwarrow & \\
\overset{w}{\bullet} &  &  &  & \overset{u_{3}}{\bullet} &  &  &  &  &  &
\overset{v}{\bullet}\\
& \nwarrow &  & \nearrow &  & \nwarrow &  &  &  & \nearrow & \\
&  & \overset{w_{2}}{\bullet} &  &  &  & \overset{u_{2}}{\bullet} & \leftarrow
& \overset{v_{2}}{\bullet} &  &
\end{array}
$
\end{example}

The graph $E_{3}$ is obtained from $E_{2}$ by adding an edge connecting the
vertex $v$ to $v_{1}$. Now the ideals $I_{(H_{1},B_{H_{1}})}$ and
$I_{(H_{1},B_{H_{1}\backslash\{v_{1}\}})}$ are both prime ideals, but
$I_{(H_{1},B_{H_{1}\backslash\{v_{1}\}})}$ is not a maximal ideal since
$I_{(H_{1},B_{H_{1}\backslash\{v_{1}\}})}\subsetneqq I_{(H_{1},B_{H_{1}})}$.
Thus the Krull dimension of $L_{K}(E_{3})$ is not $0$. Note that condition
II(c) of Theorem \ref{Prime=maximal} is not satisfied.

\begin{example}
\label{Ex 6.8}(a) Let $E_{4}$ be the graph $%
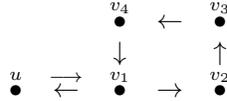
\begin{array}
[c]{ccccc}
&  & \overset{v_{4}}{\bullet} & \leftarrow & \overset{v_{3}}{\bullet}\\
&  & \downarrow &  & \uparrow\\
\overset{u}{\bullet} & \overset{\longrightarrow}{\leftarrow} & \overset{v_{1}%
}{\bullet} & \rightarrow & \overset{v_{2}}{\bullet}%
\end{array}
$. Clearly $E_{4}$ satisfies Condition (L), the MT-3 condition and has no
non-empty proper hereditary saturated subset of vertices. Hence $L_{K}(E_{4})$
is a prime simple ring and has Krull dimension $0$. The same conclusion holds
if in the graph $E_{5}$\ we replace the cycle $v_{1}v_{2}v_{3}v_{4}$ by an
arbitrary cycle $c$.

(b) Let $E_{5}$ be the graph with a single vertex and a single loop. Clearly
$E_{5}$ satisfies the Condition I \ of Theorem \ref{Prime=maximal} and
$L_{K}(E_{5})\cong K[x,x^{-1}]$, which being a Euclidean domain has all its
non-zero prime ideals maximal. Since $K[x,x^{-1}]$ is an integral domain,
$\{0\}$ is a prime ideal and $L_{K}(E_{5})$ has Krull dimension one.
\end{example}

We now construct Leavitt path algebras of any prescribed Krull dimension.

\begin{example}
\label{Pyramid Graph} For each $n\leq\omega$, there exists a Leavitt path
algebra $L_{K}(E)$ with Krull dimension $n$.
\end{example}

Let $P_{\omega}=%
{\textstyle\bigcup\limits_{n\in\mathbb{N}}}
P_{n}$ be the "Pyramid" graph of length $\omega$ constructed inductively in
\cite{abrams5} and represented pictorially as follows.%
\[%
\raisebox{-0.0415in}{\includegraphics[
natheight=2.360900in,
natwidth=4.347400in,
height=2.4016in,
width=4.3984in
]%
{LN92WS04.bmp}%
}
\]
Specifically, the graph $P_{1}$ consists of the infinite line segment in the
first row of $P_{\omega}$. The pyramid\ graph $P_{2}$ consists of the vertices
in the top two infinite line segments together with all the edges they emit in
$P_{\omega}$. More generally, for any $n\geq1$, the pyramid graph $P_{n}$
consists of all the vertices in the first $n$ "rows" of $P_{\omega}$ together
with all the edges they emanate.

(i) \textbf{Claim:} For each integer $n$, the Krull dimension of
$L_{K}(P_{n+1})$ is $n$.

Now the graph $E=P_{n+1}$ \ is row-finite and contains the chain of pyramid
subgraphs $P_{1}\subsetneqq...\subsetneqq P_{n}$ where $P_{i}$ is embedded in
$P_{i+1}$ by identifying $P_{i}$ with the top $i$ "layers" of $P_{i+1}$.
\ Observe that for each $i=1,...,n$, $(P_{i})^{0}$ is a hereditary saturated
subset of $E^{0}$and $E^{0}\backslash(P_{i})^{0}$ satisfies the MT-3 condition
and so, the ideal $J_{i}=<(P_{i})^{0}>$ is a prime ideal of $L_{K}(E)$. Also
$E^{0}$ satisfies MT-3 and so, by Theorem \ref{primeLPA}, $\{0\}$ is a prime
ideal. Thus we get a chain of prime ideals $J_{0}=\{0\}\subsetneqq
J_{1}\subsetneqq...\subsetneqq J_{n}$. Moreover, $(P_{1})^{0},...,(P_{n})^{0}$
are the only proper non-empty hereditary saturated subsets of $E^{0}$ and so
$J_{1,}...,J_{n}$ are the only\ non-zero ideals of $L_{K}(E)$. This proves
that the Krull dimension of $L_{K}(P_{n+1})$ is exactly $n$.

(ii) \ Since $P_{\omega}=\cup_{n<\omega}P_{n}$, it is then clear that
$L_{K}(P_{\omega})$ is the union of the ascending chain of prime ideals
$J_{0}=\{0\}\subsetneqq J_{1}\subsetneqq...\subsetneqq J_{n}\subsetneqq......$
where $J_{n}=<(P_{n})^{0}>$. Thus $L_{K}(P_{\omega})$ has Krull dimension
$\omega$.

(iii) \ Similarly, using the transfinite construction of the pyramid graphs
$P_{\kappa}$ for various infinite ordinals $\kappa$ as given in \cite{abrams5}%
, we can establish the existence of \ Leavitt path algebras having Krull
dimension $\kappa$ for various infinite ordinals.

\textbf{Properties of\ }$L_{K}(P_{\kappa})$: The\ Leavitt path algebra
$L_{K}(P_{\kappa})$ of the Pyramid graph $P_{\kappa}$ has many interesting
properties: (i) It is a primitive von Neumann regular ring (as $E$ is acyclic;
see \cite{AR}) (ii) Every ideal of $L_{K}(P_{\kappa})$ is a primitive ideal
and is further graded (Justification: \ $E$ is acyclic, $E^{0}$ satisfies the
MT-3 condition and every hereditary saturated subset of $E^{0}$ satisfies the
CSP with respect to the countable set of vertices in the "first layer",
namely,\ the set $\{v_{11},v_{12},v_{13},....\}$); \ (iii) $L_{K}(P_{\kappa})$
is "two-sided" uniserial (that is, its ideals form a chain); (iv) Moreover,
the chain $J_{0}=\{0\}\subsetneqq J_{1}\subsetneqq...\subsetneqq
J_{i}\subsetneqq\cdot\cdot\cdot$ where $J_{i}=<(P_{i})^{0}>$ is a "two-sided
composition series" or "saturated" in the sense that, for each $i$, $J_{i}$ is
an ideal with $J_{i+1}/J_{i}$ a simple ring;\ (v) The ideals chain of
$L_{K}(P_{\kappa})$ is the socle series for $L_{K}(P_{\kappa})$, that is,
$J_{i+1}/J_{i}=soc(L_{K}(P_{\kappa})/J_{i})$, for all $i\geq0$.

\section{Minimal Prime Ideals of $L_{K}(E)$}

In this section, we characterize those ideals of an arbitrary Leavitt path
algebra which are minimal and also construct examples of such ideals. Recall
that a prime ideal $P$ of a ring $R$ is said to be a \textit{minimal prime
ideal} if there is no prime ideal $Q$ satisfying $Q\subsetneqq P$. Recall that
for any vertex $v$ in a graph $E$, the tree of $v$ is the set $T(v)=\{w\in
E^{0}:v\geq w\}$ and $M(v)=\{w\in E^{0}:w\geq v\}$.

For convenience in expression, we introduce the following definition.

\begin{definition}
\label{Property *}A hereditary saturated subset $H$ of vertices in the graph
$E$ is said to have the Property (*), if every non-empty proper maximal tail
$S$ in $H$ contains a vertex $u$ such that $T(u)\cap T(v)=\Phi$ (the empty
set) for some $v\notin H$. $\ \ $
\end{definition}

\begin{theorem}
\label{Minimal prime}Let $E$ be an arbitrary graph and $K$ be any field. Then
a non-zero prime ideal $P$ of $L_{K}(E)$ with $P\cap E^{0}=H$ is a minimal
prime ideal if and only $P$ is a graded prime ideal such that $H$ satisfies
the Property (*) and either $P=I_{(H,B_{H\backslash\{v\}})}$ for some $v\in
B_{H}$ or $P=I_{(H,B_{H})}$ with $E^{0}\backslash H\neq M(v)$ for all $v\in
B_{H}$.
\end{theorem}

\begin{proof}
Let $P$ be a non-zero minimal prime ideal. Now $P$ cannot be a non-graded
prime ideal, since by Theorem \ref{Main Theorem}(iii), $P$ is then of the form
$P=<I_{(H,B_{H})},p(c)>$ and by Lemma \ref{Lemma 3.6}, $I_{(H,B_{H}%
)}\subsetneqq P$ will (even if $0$) always be a prime ideal. Thus $P$ must be
a graded ideal of $L_{K}(E)$. If $H$ does not have the Property (*), then
there is a maximal tail $S$ in $H$ such that for all $u\in S$ and $v\in
E^{0}\backslash H$, $T(u)\cap T(v)\neq\Phi$ and so there is a $w\in T(u)\cap
T(v)$ such that $u\geq w$ and $v\geq w$. It is then readily seen that
$M=S\cup(E^{0}\backslash H)$ is a maximal tail in $E^{0}$. Then the ideal
generated by $E^{0}\backslash M=H\backslash S$ is, by Theorem
\ref{Main Theorem}, a non-zero prime ideal of $L_{K}(E)$ contained in $P$,
thus contradicting the minimality of $P$. Also if $P$ is of the form
$I_{(H,B_{H})}$, and $M(v)=E^{0}\backslash H$ for some $v\in B_{H}$, then
$I_{(H,B_{H})}$ will properly contain the prime ideal $I_{(H,B_{H\backslash
\{v\}})}$, a contradiction to the minimality of $I_{(H,B_{H})}$.

Conversely, let $P\neq0$ be a graded prime ideal of $L_{K}(E)$ satisfying the
given hypotheses.\ By Theorem \ref{Main Theorem}, $P$ is of the form
$I_{(H,B_{H\backslash\{u\}})}$\ for some $u\in B_{H}$ or of the form
$I_{(H,B_{H})}$. We claim that there is no prime ideal $J\subsetneqq P$ such
that $J\cap E^{0}=H$. This is clear (by Theorem \ref{Main Theorem}) when $P$
is of the form $I_{(H,B_{H\backslash\{u\}})}$. On the other hand, if $P$ is of
the form $I_{(H,B_{H})}$ and if there exists such a prime ideal $J$ \ with
$J\cap E^{0}=H$, then $J$ must be of the form $I_{(H,B_{H\backslash\{v\}})}$
for some $v\in B_{H}$ in which case $M(v)=E^{0}\backslash H$, a contradiction
to the hypothesis. Suppose now that there is a non-zero prime ideal
$J\subsetneqq P$ with $J\cap E^{0}=X\subsetneqq H$. By Theorem
\ref{Main Theorem}, $M=E^{0}\backslash X$ is a maximal tail in $E^{0}$.
Clearly $S=H\cap M$ is a maximal tail in $H$. But this $S$ contradicts the
Property (*) of $H$, since, by the MT-3 condition of $M$, for any $u\in
S\subset$ $M$ and $v\in E^{0}\backslash H\subset M$ \ there is a $w\in$ $M$
\ such that $u\geq w,v\geq w$ showing that $T(u)\cap T(v)\neq\Phi$. This
proves that $P$ is a minimal prime ideal of $L_{K}(E)$
\end{proof}

\textbf{Example:}

Consider the graph $E_{1}$ of Example \ref{Ex 6.5}. The (graded) ideal $I$
generated by\ the hereditary saturated set $H_{1}=\{w,w_{1},w_{2},u_{1}%
,u_{2},u_{3}\}$ is, by the Theorem \ref{Minimal prime}, a minimal prime ideal
of $L_{K}(E_{3})$, since $S=\{w,w_{1},w_{2}\}$ is the only proper maximal tail
in $H_{1}$ and for $w_{1}\in S$ and $v\notin H_{1}$, we have $T(w_{1})\cap
T(v)=\Phi$.

For row-finite graphs Theorem \ref{Minimal prime} reduces to the following.

\begin{corollary}
\label{row-finite minimal prime}Let $E$ be a row-finite graph and $K$ be any
field. Then a non-zero prime ideal $P$ with $P\cap E^{0}=H$ is a minimal prime
ideal if and only if $P=<H>$ and $H$ is satisfies the Property (*).
\end{corollary}

\section{Height One Prime Ideals of $L_{K}(E)$}

Recall that the\textbf{\ height }of a prime ideal $P$ is $n$ if $n$ is the
largest integer such that there exists a chain of different prime ideals
$P=P_{0}\supsetneqq P_{1}\supsetneqq....\supsetneqq P_{n}$. Thus minimal prime
ideals have height $0$. Also a ring $R$ has Krull dimension $0$ if and only if
every prime ideal of $R$ has height $0$. The concept the height can also be
defined for infinite ordinals $\lambda$ in an analogous fashion. The
\textbf{height one} prime ideals play an important role in the study of
commutative rings and algebraic geometry. In this section, we describe the
height one prime ideals of an arbitrary Leavitt path algebra.

\begin{theorem}
\label{Height-1}Let $E$ be an arbitrary graph and $K$ be any field. Then a
prime ideal $P$ of the Leavitt path algebra $L_{K}(E)$ with $P\cap E^{0}=H$
\ is a prime ideal of height one if and only if $P$ has one of the following properties:

(i) \ $P$ is a non-graded prime ideal with EITHER $H=\Phi$ (the empty set) OR
$H$ satisfies Property (*) and for every $u\in B_{H}$, $M(u)$ $\neq
E^{0}\backslash H$.

(ii) $P=I_{(H,B_{H})}$, and EITHER \ there is a vertex $u\in B_{H}$ such that
$M(u)=E^{0}\backslash H$ and $H$ satisfies Property (*) in $E^{0}$ OR
$M(u)\neq E^{0}\backslash H$ for any $u\in B_{H}$ and there is exactly one
maximal tail in $E^{0}$ properly containing $E^{0}\backslash H$;

(iii) $P=I_{(H,B_{H}\backslash\{u\})}$ and there is exactly one maximal tail
in $E^{0}$ properly containing $M(u)$.
\end{theorem}

\begin{proof}
We distinguish three cases corresponding to the three types of prime ideals
indicated in Theorem \ref{Main Theorem}.

Suppose $P$ is a non-graded prime ideal. Now, by Lemmas \ref{Lemma 3.5} and
\ref{Lemma 3.6}, $I_{(H,B_{H})}$ is a prime ideal that contains every graded
ideal inside $P$. Moreover, $I_{(H,B_{H})}$ also contains any non-graded prime
ideal of $L_{K}(E)$ inside $P$. Because, if $J=<I_{(H^{\prime},B_{H^{\prime}%
})},g(c^{\prime})>$ is a non-graded prime ideal properly contained in $P$ with
$c^{\prime}$ a cycle without exits based at a vertex $v^{\prime}$ in
$E^{0}\backslash H^{\prime}$, then necessarily $H^{\prime}\subsetneqq H$ and
so $H$ contains a $w\in$ $E^{0}\backslash H^{\prime}=M(v^{\prime})$. Then the
hereditary set $H$ contains $v^{\prime}$ and hence $(c^{\prime})^{0}$ and this
implies that the ideal $I_{(H,B_{H})}$ contains $g(c^{\prime})$ and hence $J$.
Consequently, $P$ has height one if and only if $I_{(H,B_{H})}$ is a minimal
prime ideal. By Theorem \ref{Minimal prime}, this is possible if and only if
either $I_{(H,B_{H})}=\{0\}$, that is $H=\Phi$, or for every $u\in B_{H}$,
$M(u)$ $\neq E^{0}\backslash H$ and $H$ satisfies Property (*).

Suppose $P$ is a graded prime ideal of the form $I_{(H,B_{H})}$. By Theorem
\ref{Main Theorem}, $I_{(H,B_{H})}$ contains a prime ideal of the form
$I_{(H,B_{H}\backslash\{u\})}$ for some $u\in B_{H}$ exactly when
$M(u)=E^{0}\backslash H$. Thus in this case, $I_{(H,B_{H})}$ will have height
one if and only if $I_{(H,B_{H}\backslash\{u\})}$ is a minimal prime ideal.
Appealing to Proposition \ref{Minimal prime}, we then conclude that in this
case $I_{(H,B_{H})}$ has height one if and only if for some $u\in B_{H}$,
$M(u)=E^{0}\backslash H$ and $H$ satisfies the Property (*). \ Suppose on the
other hand, $M(u)\neq E^{0}\backslash H$ for any $u\in B_{H}$. If
$I_{(H,B_{H})}$ has height one, then the unique prime ideal $J\subsetneqq
I_{(H,B_{H})}$ must necessarily be a graded ideal, since every non-graded
prime ideal, by Lemma \ref{Lemma 3.6}, contains another (graded) prime ideal.
Moreover, $X=J\cap E^{0}\subsetneqq H$ since otherwise $J=I_{(H,B_{H}%
\backslash\{u\})}$ for some $u\in B_{H}$ and this contradicts our supposition
that $M(u)\neq E^{0}\backslash H$ for any $u\in B_{H}$. Thus the graded ideal
$J$ is the only prime ideal contained in $I_{(H,B_{H})}$ with $X=J\cap
E^{0}\subsetneqq H$ and this happens if and only if $M=E^{0}\backslash X$ is
the only maximal tail properly containing $E^{0}\backslash H$.

Suppose $P$ is a graded prime ideal of the form $I_{(H,B_{H}\backslash\{u\})}$
for some $u\in B_{H}$. Then $I_{(H,B_{H}\backslash\{u\})}$ has height one if
and only if there exists exactly one prime ideal $J\subsetneqq P$. As argued
in the preceding paragraph, this ideal $J$ has to be \ graded ideal with
$J\cap E^{0}=X\subsetneqq H$. Thus $J$ will be the only prime ideal contained
in $P$ if and only if $M=E^{0}\backslash X$ is the only maximal tail properly
containing $E^{0}\backslash H$.
\end{proof}

\textbf{Remark}: From the proof of Theorem \ref{Height-1} we get a slightly
sharper form of Lemmas \ref{Lemma 3.5} and \ref{Lemma 3.6} for prime ideals:
If $P=<I_{(H,B_{H})},f(c)>$ is a non-graded prime ideal of $L_{K}(E)$, then
$I_{(H,B_{H})}$ is a prime ideal that contains every other prime ideal of
$L_{K}(E)$ inside $P$.

For row-finite graphs we get the following.

\begin{corollary}
\label{row-finite height 1} Let $E$ be a row-finite graph and $K$ be any
field. Then a prime ideal $P$ of $L_{K}(E)$ with $P\cap E^{0}=H$ has height
$1$ if and only if EITHER $P$ is a non-graded ideal with $H$ empty or $H$
satisfying the Property (*), OR $P$ is a graded ideal such that there is
exactly one maximal tail containing $E^{0}\backslash H$.
\end{corollary}

\textbf{Example:} Consider the graph $G$ given below:%
\begin{figure}
[h]
\begin{center}
\includegraphics[
natheight=1.888700in,
natwidth=1.666500in,
height=1.9268in,
width=1.7028in
]%
{LN92WS05.bmp}%
\end{center}
\end{figure}

Let $c$ denote the loop based at $w$. Now $G^{0}$ satisfies the MT-3 condition
and so $\{0\}$ is a prime ideal of $L_{K}(G)$, by Theorem \ref{primeLPA}. Note
that $c$ is a cycle without exits in $G$. \ Now the non-zero proper ideals of
$L_{K}(G)$ are the graded ideal $<w>$ and the infinitely many non-graded
ideals $<p(c)>$, for various irreducible polynomials $p(x)\in K[x,x^{-1}]$. By
the above Theorem, every non-zero proper ideal of $L_{K}(G)$ is a prime ideal
of height one.

\section{Co-Height One Prime Ideals of $L_{K}(E)$}

A prime ideal $P$ of a ring $R$ is said to have \textbf{co-height }$n$ if $n$
is the largest integer such that there exists a chain of different prime
ideals \ $P=P_{0}\subsetneqq P_{1}\subsetneqq...\subsetneqq P_{n}\neq R$.
Thus, in particular, $P$ will have \textbf{co-height }$1$ if there exists a
prime ideal $Q\neq R$ such that $P\subset Q$ and there is no prime ideal $I$
such that $P\subsetneqq I\subsetneqq Q$ and no prime ideal $J$ such that
$Q\subsetneqq J\subsetneqq R$ and every other prime ideal $P^{\prime}\supset
P$ also has the same property as $Q$.

We shall describe the prime ideals of co-height $1$ by means of graphical
properties of $E$ and construct examples illustrating these properties.

\begin{theorem}
\label{co-height 1}Let $E$ be an arbitrary graph and $K$ be any field. Then a
prime ideal $P$ of the Leavitt path algebra $L_{K}(E)$ with $P\cap E^{0}=H$
\ has co-height $1$ if and only if $P$ satisfies one of the following
conditions:$\ $

(i) $\ P$ is a non-graded ideal such that all the non-empty maximal tails
properly contained in \ $E^{0}\backslash H$ satisfy the Condition (L) and form
a non-empty antichain under set inclusion;

(ii) $P$ is a graded ideal of the form $I_{(H,B_{H})}$ such that EITHER (a)
all the non-empty maximal tails properly contained in \ $E^{0}\backslash H$
satisfy the Condition (L) and form a non-empty antichain under set inclusion
OR (b) there are no non-empty maximal tails properly contained in
$E^{0}\backslash H$ and there exists a cycle without exits in $E^{0}\backslash
H$ based at a vertex $v$ such that that $E^{0}\backslash H=M(v)$;

(iii) \ $P$ is a graded ideal of the form $I_{(H,B_{H}\backslash\{u\})}$ for
some $u\in B_{H}$, such that every cycle containing $u$ has exits in
$E^{0}\backslash H$ and that there are no non-empty maximal tails properly
contained in $E^{0}\backslash H$.
\end{theorem}

\begin{proof}
By Theorem \ref{Main Theorem}, there are three types prime ideals in
$L_{K}(E)$. Accordingly, we consider three different cases.

Case (i): Let $P$ be a non-graded prime ideal, so by Theorem
\ref{Main Theorem}, $P=<I_{(H,B_{H})},f(c)>$ where $c$ is a cycle without
exits based at a vertex $v$ in $E^{0}\backslash H$ and $E^{0}\backslash
H=M(v)$.

Suppose $P$ has co-height $1$. So if $P\subset Q$ for some prime ideal $Q$,
then there is no prime ideal $I$ such that $P\subsetneqq I\subsetneqq Q$ and
no prime ideal $J$ such that $Q\subsetneqq J\subsetneqq L_{K}(E)$ and that
there is at least one such prime ideal $Q$. We claim that $Q$ cannot be a
non-graded prime ideal. Suppose, on the contrary, $Q=<I_{(H^{\prime
},B_{H^{\prime}})},g(c^{\prime})>$ where $H^{\prime}=Q\cap E^{0}$, $c^{\prime
}$ is a cycle without exits in $E^{0}\backslash H^{\prime}$. By the Remark
preceding Corollary \ref{row-finite height 1}, $I_{(H^{\prime},B_{H^{\prime}%
})}$ properly contain $P$, contradicting the fact that there is no prime ideal
$I$ with $P\subsetneqq I\subsetneqq Q$. Thus every prime ideal $Q\supsetneqq
P$ must be a graded prime ideal, say, $Q=I_{(H^{\prime},B_{H^{\prime}})}$ with
$H^{\prime}\supsetneqq H$. Clearly $v\in H^{\prime}$. Now $M=E^{0}\backslash
H^{\prime}$ is a maximal tail properly contained in $E^{0}\backslash H$. We
claim that $M$ satisfies the Condition (L). Because, otherwise, there will be
a cycle $c^{\prime\prime}$ without exits in $M$ based at a vertex $v^{\prime}$
and $M=M(v^{\prime})$. \ Then for any irreducible polynomial $p(x)\in
K[x,x^{-1}]$, the ideal $<Q,p(c^{\prime\prime})>$ will be, by Theorem
\ref{Main Theorem}, a prime ideal containing $Q$, a contradiction. If there is
a maximal tail $N$ inside $E^{0}\backslash H$ such that $M\subsetneqq N$ then
since $H\subset X=E^{0}\backslash N\subset H^{\prime}$, we have $P\subsetneqq
I=I_{(X,B_{X})}\subsetneqq Q$, a contradiction. $\ $If, on the other hand,
there is a non-empty maximal tail $N^{\prime}\subsetneqq M$ then for
$Y=E^{0}\backslash N^{\prime}$, the ideal $J=I_{(Y,B_{Y})}$ is a prime ideal
satisfying $Q\subsetneqq J\subsetneqq L_{K}(E)$, a contradiction. Hence the
non-empty maximal tails properly contained in $E^{0}\backslash H$ form an
antichain under set inclusion and satisfy the Condition (L). Note that there
is at least one such maximal tail in $E^{0}\backslash H$, since there is at
least one prime ideal $Q\supsetneqq P$.

Conversely, suppose $P$ satisfies Condition (i) of the Theorem. By hypothesis,
the set $S$ of non-empty maximal tails properly contained in $E^{0}\backslash
H$ is non-empty. Let $M$ be an arbitrary member of this set $S$ and let
$H^{\prime}=E^{0}\backslash M$. \ Then $Q=I_{(H^{\prime},B_{H^{\prime}})}$ is
a prime ideal containing $P$. If there is a prime ideal $I$ such that
$P\subsetneqq I\subsetneqq Q$ then for $X=I\cap E^{0}$, we have $H\subsetneqq
X\subsetneqq H^{\prime}$ and so $N=E^{0}\backslash X$ is a maximal tail that
satisfies $M\subsetneqq N\subsetneqq E^{0}\backslash H$, thus contradicting
the hypothesis. \ Suppose there is a prime ideal $J$ with $Q\subsetneqq
J\subsetneqq L_{K}(E)$. Let $J\cap E^{0}=Y$. We claim $Y\neq H^{\prime}$.
Indeed if $Y=H^{\prime}$, then first of all $J$ cannot be a graded ideal since
then $B_{Y}=B_{H^{\prime}}$ and $J=I_{(Y,B_{Y})}=I_{(H^{\prime},B_{H^{\prime}%
})}=Q$, a contradiction. On the other hand if $J$ were a non-graded prime
ideal, then there must be a cycle without exits in $E^{0}\backslash
Y=E^{0}\backslash H^{\prime}$, again a contradiction since $E^{0}\backslash
H^{\prime}$ satisfies Condition (L). Thus $Y\supsetneqq H^{\prime}$. But then
the maximal tail $E^{0}\backslash Y$ satisfies $E^{0}\backslash Y\subsetneqq
M$, a contradiction to the hypothesis. This shows that $P$ is a prime ideal of
co-height $1$. \ 

Case (ii): Let $P$ be a graded prime ideal of the form $I_{(H,B_{H})}$.

Suppose $P$ has co-height $1$, so there is a prime ideal $Q\supset P$ such
that there is no prime ideal $I$ such that $P\subsetneqq I\subsetneqq Q$ and
no prime ideal $J$ such that $Q\subsetneqq J\subsetneqq L_{K}(E)$. Let $Q\cap
E^{0}=H^{\prime}$. \ 

If $Q$ is a non-graded prime ideal, say, $Q=<I_{H^{\prime},B_{H^{\prime}}%
)},p(c)>$, then necessarily, $H^{\prime}=H$ since otherwise, by Lemma 3.6,
$I_{(H^{\prime},B_{H^{\prime}})}$ will be a prime ideal that satisfies
$P\subsetneqq I_{(H^{\prime},B_{H^{\prime}})}\subsetneqq Q$. Thus
$Q=<I_{(H,B_{H})},p(c)>$, with $c$ a cycle without exits in $E^{0}\backslash
H$ based at a vertex $v$ and $E^{0}\backslash H=M(v)$. Also $E^{0}\backslash
H$ cannot contain any proper non-empty maximal tail $M$ since otherwise
$X=E^{0}\backslash M$ will be a hereditary saturated subset containing $v$ and
hence $c^{0}$, and so the ideal $I_{(X,B_{X})}$ is a prime ideal containing
$Q$, a contradiction. This shows that $P$ satisfies Condition (ii)(b).

Suppose $Q$ be a graded prime ideal. By Theorem \ref{Main Theorem}0,
$Q=I_{(H^{\prime},B_{H^{\prime}})}$ or $I_{(H^{\prime},B_{H^{\prime}%
}\backslash\{u\}}$ for some $u\in B_{H^{\prime}}$. Since $I_{(H^{\prime
},B_{H^{\prime}})}\supsetneqq I_{(H^{\prime},B_{H^{\prime}}\backslash\{u\}}$
and there is no prime ideal \ properly containing $Q$, $Q\neq$ $I_{(H^{\prime
},B_{H^{\prime}}\backslash\{u\}}$. Thus $Q=I_{(H^{\prime},B_{H^{\prime}})}$
and $H^{\prime}\neq H$ since, otherwise, $Q=I_{(H,B_{H})}=P$, a contradiction.
Now $M=E^{0}\backslash H^{\prime}$ is a maximal tail properly contained in
$E^{0}\backslash H$ . If $N$ is a non-empty maximal tail and $N\subsetneqq M$
and $X=E^{0}\backslash N$, then $I_{(X,B_{X})}$ will be a prime ideal
satisfying $Q\subsetneqq I_{(X,B_{X})}\subsetneqq L_{K}(E)$, a contradiction.
Also, if there is a maximal tail $N^{\prime}$ with $M\subsetneqq N^{\prime
}\subsetneqq E^{0}\backslash H$, then for $Y=E^{0}\backslash N^{\prime}$, the
prime ideal $I_{(Y,B_{Y})}$ satisfies $P\subsetneqq I_{(Y,B_{Y})}\subsetneqq
Q$, a contradiction. Also, if the maximal tail $M$ does not satisfy Condition
(L), then there will be a cycle without exits based at a vertex $v$ in $M$ so
that $M=M(v)$. Then for an irreducible polynomial $f(x)\in K[x,x^{-1}]$, the
prime ideal $<Q,f(c)>$ properly contains $Q$, a contradiction. Hence $M$
satisfies Condition (L). We thus have shown that $P$ satisfies Condition (ii)a.

Conversely,\ suppose the prime ideal $P=I_{(H,B_{H})}$ satisfies the stated
properties in Condition (ii) of the Theorem.

Specifically, assume Condition (ii)b so that there is a cycle $c$ without
exits in $E^{0}\backslash H$ based at a vertex $v$, $E^{0}\backslash H=M(v)$
and $E^{0}\backslash H$ contains no non-empty proper maximal tails. Then\ for
some irreducible polynomial $f(x)\in K[x,x^{-1}]$, $Q=<I_{(H,B_{H})},f(c)>$ is
a prime ideal containing $P$. Clearly, there is no ideal $I$ such that
$P\subsetneqq I\subsetneqq Q$. Now there cannot be a prime ideal $J$, with
$Q\subsetneqq J\subsetneqq L_{K}(E)$. Because, if $Y=J\cap E^{0}$, first of
all $Y\neq H$ since, otherwise, $J$ has to be a non-graded prime ideal of the
form $J=<I_{(H,B_{H})},g(c)>$ for some irreducible polynomial $g(x)\in
K[x,x^{-1}]$ and this is impossible by the Remark at the end of Section 3. But
if $Y\supsetneqq H$ then $M=E^{0}\backslash Y$ will be a maximal tail properly
contained in $E^{0}\backslash H$, a contradiction. This shows that $P$ has
co-height $1$.

Suppose now that Condition (ii)a holds. By hypothesis, the set $S$ of
non-empty maximal tails properly contained in $E^{0}\backslash H$ is
non-empty. Now repeat the proof of the converse in Case (i) above\ verbatim to
conclude that $P$ has co-height $1$.

Case (iii): Let $P$ be a graded prime ideal of the form $P=$ $I_{(H,B_{H}%
\backslash\{u\})}$. Since $E^{0}\backslash H=M(u)$ satisfies the MT-3
condition, $I_{(H,B_{H})}$ will always be a prime ideal containing $P$. It is
readily seen that there is no prime ideal $I$ with $I_{(H,B_{H}\backslash
\{u\})}\subsetneqq I\subsetneqq I_{(H,B_{H})}$. So $P$ will have co-height $1$
exactly when there are no prime ideals $P^{\prime}\supsetneqq I_{(H,B_{H})}$.
If a prime ideal $P^{\prime}\supsetneqq I_{(H,B_{H})}$ with $X=P^{\prime}\cap
E^{0}$, then either $X=H$ or $X\supsetneqq H$. If $X=H$, then $P^{\prime}$
must be non-graded and so, by Theorem \ref{Main Theorem}, $E^{0}\backslash
H=E^{0}\backslash X$ will have a cycle without exits. This is possible if and
only if $E^{0}\backslash H$ does not satisfy Condition (L). On the other hand,
$X\supsetneqq H$ if and only if $E^{0}\backslash X$ is a proper non-empty
maximal tail in $E^{0}\backslash H$. This proves that $I_{(H,B_{H}%
\backslash\{u\})}$ has co-height $1$ if and only if $E^{0}\backslash H$
satisfies the condition (L) and contains no non-empty proper maximal tails.
\end{proof}

\begin{example}
\label{Ex co-ht-1}1. Let $E$ be the graph
\begin{figure}
[h]
\begin{center}
\includegraphics[
natheight=1.364700in,
natwidth=1.187400in,
height=1.4001in,
width=1.2211in
]%
{Example37.jpg}%
\end{center}
\end{figure}

\end{example}

Let $c$ denotes the loop at the vertex $v$. Clearly $E^{0}$ satisfies the MT-3
condition and so, by Theorem \ref{Main Theorem}, the ideal $P=<v-c>$ is a
non-graded prime ideal of $L_{K}(E)$. Now $H=\{u,v\}$ is a hereditary
saturated subset and $P\subset Q=<H>$ which is a maximal ideal of $L_{K}(E)$,
since $L_{K}(E)/Q\cong L_{K}(E^{0}\backslash H)\cong$ $L(1,2)$, a Leavitt
algebra which is a simple ring (see \cite{Abrams1}). Thus $P$ is a co-height
one non-graded prime ideal. Note that condition (i) of the above theorem is
trivially satisfied.

2. In Example \ref{Ex 4.6}, the ideal $Q=I_{(H,B_{H})}$ is the only co-height
$1$ prime ideal of $L_{K}(E)$ and is graded. Note that Condition (ii) of the
above theorem holds for $E^{0}\backslash H$.

3. Consider the Example \ref{Ex 4.5}. For each $j\geq1$, the prime ideal
$P_{j}$ has co-height $1$. Note that Condition (iii) of the above Theorem
holds. It is interesting to note that each $P_{j}$ is also a height $1$ prime ideal.

\section{Prime homomorphic images of $L_{K}(E)$\textbf{.}}

Finally, we consider the prime homomorphic images of a Leavitt path algebra
$L_{K}(E)$. Should they all be again Leavitt path algebras ? \ Fruitful
correspondence with Ken Goodearl resulted in a \ definitive answer to this
question in the case of a finite graph $E$ and this appears as Proposition
4.4\ in \cite{Abrams2}. As an application\ of Theorem \ref{Main Theorem}, we
get a complete description of the prime homomorphic images of $L_{K}(E)$ for
arbitrary graphs $E$. The same proof shown in \cite{Abrams2}, with minor
modifications, works for arbitrary graphs $E$. We outline the proof for the
sake of completeness.

\begin{proposition}
\label{Hom image}Let $E$ be an arbitrary graph and let $P$ be a prime ideal of
$L_{K}(E)$. Then either $T=L_{K}(E)/P$ is isomorphic to a Leavitt path algebra
or $T/Soc(T)$ is isomorphic to a Leavitt path algebra. In the latter case,
$Soc(T)$ is a simple ring being a direct sum of isomorphic simple left ideals
of $T$.
\end{proposition}

\begin{proof}
If $P=I_{(H,S)}$ is a graded ideal, then, by \cite{Tomforde 20},
$T=L_{K}(E)/P\cong L_{K}(E\backslash(H,S))$. Suppose $P$ is a non-graded prime
ideal. By Theorem \ref{Main Theorem}, $P=<I_{(H.B_{H})},f(c)>$ where $c$ is a
unique cycle without K in $E$ based at a vertex $v$ and $f(x)$ is an
irreducible polynomial in $K[x,x^{-1}]$. Note that $v\notin B_{H}$. As pointed
out in the proof of Theorem \ref{Primitive Ideal Thm} (i), for the idempotent
$\bar{v}=v+P$, we have $\bar{v}T\bar{v}$ is a field. This implies that
$T\bar{v}$ is then a simple left ideal of $T$ (by Proposition 1, Chapter 4,
Page 65 in \ \cite{Jacobson 15} (observing that the prime ring $T$ has no
non-zero nilpotent one-sided ideals)) . Thus $T\bar{v}\subset Soc(T)$. Since
$T$ is a prime ring, its socle is a direct sum of isomorphic simple left
ideals and, in particular, is the two-sided ideal generated by any simple left
ideal in it. We then conclude that $Soc(T)=<\bar{v}>=(<v>+P)/P$, the ideal
generated by $\bar{v}$. Now $P+<v>=I_{(H,B_{H})}+<v>=J$ is a graded ideal
(being the sum of two graded ideals). Thus $T/Soc(T)=(L_{K}%
(E)/P)/(P+<v>/P)\cong L_{K}(E)/(P+<v>)$ which is isomorphic to the Leavitt
path algebra $L_{K}(E\backslash(H^{\prime},S^{\prime}))$ where $H^{\prime
}=J\cap E^{0}$ and $S^{\prime}=\{w\in B_{H^{\prime}}:w^{H^{\prime}}\in J\}$.
\end{proof}

REMARK: In the above Proposition, if $E^{0}$ (or more generally, $(E\backslash
H,S))^{0}$) is countable, then $Soc(T)$ will be a direct sum of countably many
isomorphic simple left ideals and, in this case, $Soc(T)\cong L_{K}(F)$ where
$F$ is the infinite straight line graph \ $\overset{v_{1}}{\bullet}%
\rightarrow\overset{v_{2}}{\bullet}\rightarrow\overset{v_{3}}{\bullet
}\rightarrow\bullet\rightarrow\bullet\cdot\cdot\cdot\cdot$ . So $T$ can then
be realized as an extension of a Leavitt path algebra by another Leavitt path
algebra. Moreover, define a new graph $G$ by forming the disjoint union of the
graphs $F$ and $E\backslash(H^{\prime},S^{\prime})$ (which was defined in the
proof of Proposition \ref{Hom image}) and connecting each line point in
$E\backslash(H^{\prime},S^{\prime})$ to the vertex $v_{1}$ of the graph $F$ by
an edge. Then in the graph $G$, the line points are precisely the vertices of
the graph $F$ and the quotient graph $G\backslash F$ is the same as the graph
$E\backslash(H^{\prime},S^{\prime})$. Then the Leavitt path algebra $L_{K}(G)$
has the property that $Soc(L_{K}(G))\cong Soc(T)$ and $L_{K}(G)/Soc(L_{K}%
(G))\cong T/Soc(T)$.

\textbf{Acknowledgement}: The author thanks Chris Smith for his help in the
proof of Corollary \ref{Condition K}.

\end{document}